\documentclass[10pt]{amsart}
\usepackage{amsfonts,amssymb,amsmath,amsthm}
\usepackage[footskip=45pt]{geometry}
\usepackage{enumerate}
\usepackage[english]{babel}
\usepackage{theoremref}
\usepackage{tikz}
\usepackage{tikz-cd}
\usepackage{url}
\usepackage{verbatim} 
\usepackage{amsrefs}
\usepackage{bookmark}
\bookmarksetup{
  open,
  numbered,
  addtohook={%
    \ifnum\bookmarkget{level}>3 %
      \renewcommand*{\numberline}[1]{}%
    \fi
  },
}

\definecolor{darkblue}{rgb}{0.0,0.0,0.3}
\hypersetup{colorlinks,breaklinks,
            linkcolor=darkblue,urlcolor=darkblue,
            anchorcolor=darkblue,citecolor=darkblue}
\allowdisplaybreaks[1]

\author{HAO LEE}
\address{
	Department of Mathematics\\
	University of Chicago\\
	Chicago, Illinois}
\email{haolee@math.uchicago.edu}

\begin{document}

\title{Towards the Derived Jacquet-Emerton Module Functor}


\global\long\def\a{\mathbb{A}}
\global\long\def\c{\mathbb{C}}
\global\long\def\C{\mathcal{C}}
\global\long\def\cp{\c_{p}}
\global\long\def\d{\partial}
\global\long\def\eps{\epsilon}
\global\long\def\f{\mathbb{F}}
\global\long\def\F{\mathcal{F}}
\global\long\def\fp{\mathbb{F}_{p}}
\global\long\def\fq{\mathbb{F}_{q}}
\global\long\def\g{\mathfrak{g}}
\global\long\def\G{\mathbb{G}}
\global\long\def\h{\mathbb{H}}
\global\long\def\H{\mathfrak{h}}
\global\long\def\l{\mathcal{L}}
\global\long\def\L{\mathbb{L}}
\global\long\def\n{\mathbb{N}}
\global\long\def\N{\mathfrak{n}}
\global\long\def\m{\mathbb{M}}
\global\long\def\M{\mathfrak{m}}
\global\long\def\p{\mathfrak{p}}
\global\long\def\P{\mathbb{P}}
\global\long\def\q{\mathbb{Q}}
\global\long\def\qp{\q_{p}}
\global\long\def\o{\mathcal{O}}
\global\long\def\ok{\mathcal{O}_{K}}
\global\long\def\t{\mathbb{T}}
\global\long\def\u{\mathfrak{u}}
\global\long\def\w{\omega}
\global\long\def\z{\mathbb{Z}}
\global\long\def\zp{\z_{p}}
\global\long\def\sub{\subseteq}
\global\long\def\Ad{\textnormal{Ad}}
\global\long\def\isom{\cong}
\global\long\def\inj{\hookrightarrow}
\global\long\def\surj{\twoheadrightarrow}
\global\long\def\map{\mapsto}
\global\long\def\la{\leftarrow}
\global\long\def\ra{\rightarrow}
\global\long\def\gal{\textnormal{Gal}}
\global\long\def\lp{\llbracket}
\global\long\def\rp{\rrbracket}
\global\long\def\ind{\mbox{Ind}}
\global\long\def\da{\downarrow}
\global\long\def\ua{\uparrow}
\global\long\def\gqp{G_{\q_{p}}}
\global\long\def\ord{\textnormal{ord}}
\global\long\def\an{\textnormal{an}}
\global\long\def\An{\textnormal{-an}}
\global\long\def\La{\textnormal{la}}
\global\long\def\rep{\textnormal{Rep}}
\global\long\def\fs{\textnormal{fs}}
\global\long\def\vn{(V_{\h_{n}^{\circ}\textnormal{-an}})^{\prime}_{b}}
\global\long\def\vln{(V_{\h(z^{-1})_{n}^{\circ}\textnormal{-an}})^{\prime}_{b}}
\global\long\def\vhn{(V_{\h(z)_{n}^{\circ}\textnormal{-an}})^{\prime}_{b}}
\global\long\def\dn{D(H_{0},\h_{n}^{\circ})}
\global\long\def\res{\textnormal{res}}
\global\long\def\im{\textnormal{im}}
\global\long\def\ind{\textnormal{Ind}}
\global\long\def\pol{\textnormal{pol}}
\global\long\def\Ext{\textnormal{Ext}}
\global\long\def\Tor{\textnormal{Tor}}
\global\long\def\sm{\textnormal{sm}}
\global\long\def\sym{\textnormal{Sym}}
\global\long\def\alg{\textnormal{alg}}
\global\long\def\bal{\textnormal{bal}}


\newtheorem{theorem}{Theorem}[section]
\newtheorem{lemma}[theorem]{Lemma}
\newtheorem{proposition}[theorem]{Proposition}
\newtheorem{corollary}[theorem]{Corollary}
\newtheorem{conjecture}[theorem]{Conjecture}
\theoremstyle{definition}
\newtheorem{definition}[theorem]{Definition}
\newtheorem{claim}[theorem]{Claim}
\newtheorem{remark}[theorem]{Remark}
\numberwithin{equation}{section}
\newtheorem{example}[theorem]{Example}

\begin{abstract}
Let $G$ be a $p$-adic Lie group associated to associated to a connected reductive group over $\qp$. Let $P$ 
be a parabolic subgroup of $G$ and let $M$ be a Levi quotient of $P$. In this paper, we define a $\delta$-functor $H^{\star}J_{P}$ from the category of admissible locally analytic $G$-representations to the category of essentially admissible locally analytic $M$-representations that extends the Jacquet-Emerton module functor $J_{P}$ defined in \cite{Jacquet1}. 
\end{abstract}

\maketitle

\tableofcontents

\vspace{3mm}


\section{Introduction}\label{sec:intro}
Let $\mathbb{G}$ be a connected reductive algebraic group over $\qp$. Let $\mathbb{P}$ be a parabolic subgroup of $\mathbb{G}$ with unipotent radical $\mathbb{N}$ and let $\m$ be a Levi quotient of $\mathbb{P}$. Let $G=\mathbb{G}(\qp)$, $P=\mathbb{P}(\qp)$, $N=\mathbb{N}(\qp)$ and $M=\m(\qp)$. In \cite{Jacquet1}, Emerton introduced the theory of the Jacquet-Emerton module functor $J_{P}$, which extends the theory of the classical Jacquet module functor on smooth $G$-representations to essentially admissible locally analytic $G$-representations. The goal of this paper is to introduce a $\delta$-functor $H^{\star}J_{P}$ extending $J_{P}$ on the category of admissible locally analytic $G$-representations, which we hope to be universal, in some suitable sense. 

Some of the difficulties with such a theory of derived Jacquet-Emerton modular functor lie in the deficiencies in the various important subcategories of locally analytic $G$-representations one might want to consider. This is also why the notion of universality is difficult to pin down. For example, the full subcategories of essentially admissible and admissible $G$-representations are abelian \cite{ST_dist}, but do not have enough injectives. Meanwhile, the full subcategory of strongly admissible $G$-representations is not abelian. We attempt to remedy this problem by computing homology groups in the category of coherent modules on certain distribution algebras, which has the benefit that all linear maps are automatically continuous and strict (see \cite[Appendix A]{Jacquet2}). In this way, we can retain various topological properties while working in the context of abstract homological algebra. 

More specifically, let $V$ be an admissible locally analytic $G$-representation over a finite extension $K$ of $\qp$. Fix a compact open subgroup $N_{0}$ of $N$, and $H_{0}$ of $G$ containing $N_{0}$. Let $\{H_{n}\}_{n\geq 0}$ be a choice of decreasing sequence of rigid analytic open normal subgroups of $H_{0}$ satisfying some nice properties. In particular, we assume that these subgroups form a basis of open neighbourhoods of the identity. For each $n\geq 0$, let $\h_{n}$ denote the rigid analytic group underlying $H_{n}$ in $\G$ and let $\h_{n}^{\circ}=\cup_{m>n}\h_{m}$. Let $V_{\h_{n}^{\circ}\An}$ be the subset of $V$ consisting of $\h_{n}^{\circ}$-analytic vectors. There is then a topological isomorphism $\underset{\underset{n}{\ra}}{\lim} V_{\h_{n}^{\circ}}\cong V$. 

Define $D(H_{0},\h_{n}^{\circ})$ to be the strong dual of $C^{\La}(H_{0},K)_{\h_{n}^{\circ}\An}$, which is the set of $K$-valued locally analytic functions of $H_{0}$ that are analytic on $\h_{n}^{\circ}$. For a well-chosen sequence of rigid analytic subgroups $\{H_{n}\}_{n\geq 0}$, the distribution algebras $D(H_{0},\h_{n}^{\circ})$ are coherent rings of compact type. The strong duals $(V_{\h_{n}^{\circ}\An})^{\prime}_{b}$ of $V_{\h_{n}^{\circ}\An}$ are finitely presented $D(H_{0},\h_{n}^{\circ})$-modules. Furthermore, finitely presented modules over $D(H_{0},\h_{n}^{\circ})$-modules have a unique structure of a compact type space making it a topological $D(H_{0},\h_{n}^{\circ})$-module, and the $D(H_{0},\h_{n}^{\circ})$-linear maps between finitely presented modules are automatically continuous and strict with respect to this topology \cite[Prop. A.10]{Jacquet2}. 

We define $H_{\star}(N_{0},(V_{\h_{n}^{\circ}\An})^{\prime}_{b})$ to be the usual group homology of $N_{0}$-coinvariants computed in the abelian category of abstract coherent $D(H_{0},\h_{n}^{\circ})$-modules. Let $H^{\star}(N_{0},V_{\h_{n}^{\circ}\An})$ be the strong dual of $H_{\star}(N_{0},(V_{\h_{n}^{\circ}\An})^{\prime}_{b})$. Define $H^{\star}J_{P}(V)$ to be \begin{equation*}
    H^{\star}J_{P}(V)= (\underset{\underset{n}{\ra}}{\lim} H^{\star}(N_{0},V_{\h_{n}^{\circ}\An}))_{\textnormal{fs}},
\end{equation*}
where $(\cdot)_{\textnormal{fs}}$ denotes the finite slope part functor introduced in \cite{Jacquet1}. The family of functors $H^{\star}J_{P}$ is a $\delta$-functor from the category of admissible locally analytic $G$-representations to the category of essentially admissible locally analytic $G$-representations. In degree zero, $H^{0}J_{P}(V)= (V^{N_{0}})_{\textnormal{fs}}$, which is exactly the definition of the Jacquet-Emerton module functor. Additionally, we prove in Theorem \ref{thm:lie_cohom} that there is a natural isomorphism \begin{equation}\label{eq:intro_lim_n_cohom}
    \underset{\underset{n}{\ra}}{\lim} H^{\star}(N_{0},V_{\h_{n}^{\circ}\An}) \cong H^{\star}(\mathfrak{n},V)^{N_{0}},
\end{equation} where $H^{\star}(\mathfrak{n},V)$ is the Lie algebra cohomology. This suggests that our construction is natural. 

Using the formula (\ref{eq:intro_lim_n_cohom}), we demonstrate how one can more easily compute the derived Jacquet-Emerton modules of Orlik-Strauch representations in the case $G=SL_{2}(\qp)$. Suppose $M$ is an object in the category $\o^{\overline{\mathfrak{p}}}$ and $\psi$ is a smooth character of the torus $T$. The Orlik-Strauch representation $\F_{\overline{P}}^{G}(M,\psi)$ is a closed subspace of a locally analytic induction from $\overline{P}$ to $G$ of a representation of $\overline{P}$. As such, we can view elements of $\F_{\overline{P}}^{G}(M,\psi)$ as locally analytic functions on $G$ satisfying some conditions. Let $\F_{\overline{P}}^{G}(M,\psi)(N)$ be the subspace of functions supported on $\overline{P}N$ and let $\F_{\overline{P}}^{G}(M,\psi)_{e}$ be the stalk of $\F_{\overline{P}}^{G}(M,\psi)(N)$ at the identity as defined in \cite{Jacquet2}. The key fact is that \begin{align*}
    H^{i}(\N,\F_{\overline{P}}^{G}(M,\psi)(N)) &= H^{i}(\N,\F_{\overline{P}}^{G}(M,\psi)^{\text{lp}}(N)) \text{ and}\\
    H^{i}(\overline{\N},\F_{\overline{P}}^{G}(M,\psi)_{e}) &= H^{i}(\overline{\N},\F_{\overline{P}}^{G}(M,\psi)^{\text{pol}}_{e})
\end{align*} for each $i$, where the superscript lp denotes the locally polynomial part, and the superscript pol denotes the polynomial part (see \cite{Jacquet2}). This can be seen as a generalization of the $I_{\overline{P}}^{G}$ functor introduced in \cite{Jacquet2}. In the proof of \cite[Prop. 4.2]{Breuil_socle2}, Breuil proved that there are isomorphisms \begin{align*}
    \F_{\overline{P}}^{G}(M,\psi)^{\text{lp}}(N) & \cong \hom_{K}(M,K)^{\N_{\p}^{\infty}}\otimes_{K}C_{c}^{\text{sm}}(N,\psi) \text{ and}\\
    \F_{\overline{P}}^{G}(M,\psi)^{\text{pol}}_{e} & \cong \hom_{K}(M,K)^{\N_{\p}^{\infty}}\otimes_{K}\psi.
\end{align*} The dual vectors of $M$ generate the space $\hom_{K}(M,K)^{\N_{\p}^{\infty}}$. In summary, these observations allow us to reduce the problem of computing the $\N$ and $\overline{\N}$ cohomology of the Orlik-Strauch representation $\F_{\overline{P}}^{G}(M,\psi)$ to computing the $\N$ and $\overline{\N}$ cohomology of the dual of $M$ in the category $\o^{\overline{\mathfrak{p}}}$, which is often simple. It seems likely that this formalism should hold for general $p$-adic reductive groups as well.

\textit{The arrangement of the paper.} In Section $2$, we recall the definition of some important distribution algebras introduced in \cite{ST_dist} and \cite{analytic}. The main results are Lemma \ref{lem:Hd_coinv} and Theorem \ref{thm:flat}. In Section 3, we define the $\delta$-functor $H^{\star}J_{P}(V)$. Section 4 presents the proof that $H^{\star}J_{P}(V)$ is a $\delta$-functor from the category of admissible locally analytic $G$-representations to the category of essentially admissible locally analytic $M$-representations. In section 5, we explicitly compute the derived Jacquet-Emerton modules of some Orlik-Strauch representations in the case of $G=SL_{2}(\qp)$. In the final section, we compute some extension classes using derived Jacquet-Emerton modules and the adjunction formula of \cite[Lem. 0.3]{Jacquet2}. 

\textit{Notation and conventions.} We fix some notation that will be used throughout the whole paper. Let $p$ be a fixed prime number. Let $\mathbb{G}$ be a connected reductive linear algebraic group over $\qp$. Let $\mathbb{P}$ be a parabolic subgroup of $\mathbb{G}$, and $\overline{\mathbb{P}}$ a choice of opposite parabolic. Let $\mathbb{N}$ and $\overline{\mathbb{N}}$ be their respective unipotent radicals. Let $\mathbb{M}=\mathbb{P}\cap\overline{\mathbb{P}}$, which is a lift of the Levi quotient of $\mathbb{P}$. Let $G=\mathbb{G}(\qp)$, $P=\mathbb{P}(\qp)$, $\overline{P}=\overline{\mathbb{P}}(\qp)$, $N=\mathbb{N}(\qp)$, $\overline{N}=\overline{\mathbb{N}}(\qp)$ and $M=\mathbb{M}(\qp)$. Let $\g$, $\N$, $\overline{\N}$ and $\M$ denote the Lie algebras associated to $G$, $N$, $\overline{N}$ and $M$ respectively. Let $K$ be a finite extension of $\qp$.

\textit{Acknowledgements.}  I would like to thank my advisor Matthew Emerton for introducing me to this topic and for his continuous support throughout the project. I would also like to thank Weibo Fu, Brian Lawrence and Gal Porat for their insightful comments and for their careful reading of previous versions of this paper. This work was partially supported by the NSERC PGS-D program.  

\vspace{3mm}
\section{Preliminaries and Distribution Algebras}\label{sec:dist}

In this section, we present the definition of various distribution algebras and their explicit descriptions following \cite{analytic, ST_dist}. Then we prove two main results: that the coinvariants and the Hausdorff coinvarince of certain distribution algebras by an open compact subgroup of $N$ coincides; as well as showing that a natural morphism between distribution algebras is flat. We continue to use the notation established at the end of section \ref{sec:intro}.

Let $\g$ denote the Lie algebra of $G$. Let $\H$ be a $\zp$-lattice in $\g$, which is a free $\zp$-module of finite rank that spans $\g$ over $\qp$. By replacing $\H$ by $c\H$ for some $c\in \q_{p}^{\times}$ of large enough valuation, we can assume that the Baker-Campbell-Hausdorff formula converges on $\H$. This defines a rigid analytic group $\h$ \cite[LG Ch. V, \textsection 4]{Serre_CBH}. Let $H=\h(\qp)$; then there is an embedding of $\qp$-analytic groups $H\ra G$, realizing $H$ as an analytic open subgroup of $G$. The subgroups $H$ of $G$ arising in this way are called good analytic open subgroups. They form a basis of open neighbourhoods of the identity of $G$. 

For the purpose of this paper, we want to work with a decreasing sequence $\{H_{n}\}_{n\geq 0}$ of good analytic open subgroups satisfying Proposition  $4.1.6$ of \cite{Jacquet1}, which we will now describe. Suppose $H_{n}$ is a good analytic open subgroup of $G$ and let $\H_{n}$ be the corresponding $\zp$-sublattice of $\g$. Let $\N_{n}=\H_{n}\cap\N$, $\M_{n}=\H_{n}\cap\M$ and $\overline{\N}_{n}=\H_{n}\cap\overline{\N}$. Let $N_{n}=H_{n}\cap N$, $M_{n}=H_{n}\cap M$ and $\overline{N}_{n}=H_{n}\cap \overline{N}$. Let $\n_{n}$, $\m_{n}$ and $\overline{\n}_{n}$ be their respective rigid analytic Zariski closures in $\h_{n}$. 

\begin{definition}
The good analytic open subgroup $H_{n}$ is said to admit a rigid analytic Iwahori decomposition with respect to $\P$ and $\overline{\P}$ if:
\begin{enumerate}
    \item Under the exponential map $\H_{n}\ra \h_{n}$, the rigid analytic groups $\n_{n}$, $\m_{n}$ and $\overline{\n}_{n}$ are identified with $\N_{n}$, $\M_{n}$ and $\overline{\N}_{n}$ respectively. Additionally, the $\qp$-points of $\n_{n}$, $\m_{n}$ and $\overline{\n}_{n}$ are exactly $N_{n}$, $M_{n}$ and $\overline{N}_{n}$ respectively. 
    \item The multiplication map \[
       \overline{\n}_{n} \times \m_{n} \times  \n_{n} \ra \h_{n}
    \] induces a rigid analytic isomorphism. 
\end{enumerate}
\end{definition}

Let $\mathbb{Z}_{\mathbb{G}}$ and $\mathbb{Z}_{\mathbb{M}}$ be the centres of $\mathbb{G}$ and $\mathbb{M}$ respectively. Let $Z_G = \mathbb{Z}_{\mathbb{G}}(\qp)$ and $Z_M = \mathbb{Z}_{\mathbb{M}}(\qp)$. Let \[
    M^{+}=\left\{m\in M:\,mN_{0}m^{-1}\sub N_{0}\right\}.
\] Let $Z_{M}^{+} = M^{+}\cap Z_{M}$, which is a submonoid of $Z_{M}$ that generates $Z_{M}$ as a group \cite[Cor. 3.3.3]{Jacquet1}. 

\begin{proposition}\cite[Prop. 4.1.6]{Jacquet1}\label{prop:good_grps}
There is a decreasing sequence of good analytic open subgroups $\{H_{n}\}_{n\geq 0}$ of $G$ satisfying:
\begin{enumerate}
    \item The subgroups form a basis of open neighbourhoods of the identity. 
    \item Suppose $H_{0}$ is obtained from exponentiating a $\zp$-lattice $\mathfrak{h}$, then $H_{n}$ is obtained from exponentiating $p^{n}\H$ for all $n\geq 0$. 
    \item For all $n\geq 0$, the natural map $\h_{n+1}\ra \h_{n}$ induced by the inclusion $H_{n+1}\ra H_{n}$ is relatively compact.
    \item Each of the subgroups $H_{n}$ is normal inside $H_{0}$.
    \item Each of the subgroups $H_{n}$ admits a rigid analytic Iwahori decomposition with respect to $\P$ and $\overline{\P}$. 
    \item Suppose $z\in Z_{M}$ and $zN_{0}z^{-1}\sub N_{0}$ (ie. $z\in Z_{M}^{+}$), then $zN_{n}z^{-1}\sub N_{n}$ for all $n\geq 0$.
    \item Suppose $z\in Z_{M}$ and $z^{-1}\overline{N}_{0}z\sub \overline{N}_{0}$, then $z^{-1}\overline{N}_{n}z\sub \overline{N}_{n}$ for all $n\geq 0$. 
    \item There exists $z\in Z_{M}$ such that $z^{-1}\overline{N}_{0}z\sub \overline{N}_{0}$ and $zN_{0}z^{-1}\sub N_{0}$, and for each $n\geq 0$, the inclusions $z^{-1}\overline{N}_{n}z\sub\overline{N}_{n}$ given by condition (7) factor through $\overline{N}_{n+1}\sub\overline{N}_{n}$. 
\end{enumerate}
\end{proposition}


We now describe some distribution rings introduced in \cite{analytic}, \cite{ST_dist} and \cite{ST_GL2} along with some important properties. Let $\h$ be a rigid analytic subgroup of $\mathbb{G}$ and let $H=\h(\qp)$. Let $C^{\an}(\h,K)$ be the space of $K$-valued rigid analytic functions on $\h$, and let $C^{\La}(H,K)$ be the set of $K$-valued locally analytic functions on $H$ as defined in \cite{analytic}. Equip these spaces with the right regular action of $H$, which is a locally analytic action. 

Suppose $\{H_{n}\}_{n\geq 0}$ is a decreasing sequence of good analytic open subgroups of $G$ satisfying Proposition \ref{prop:good_grps}. For each $n\geq 0$, let $\h^{\circ}_{n}=\bigcup_{m>n} \h_{m}$, which is a strictly $\sigma$-affinoid rigid analytic open subgroup of $\mathbb{G}$. 

Let $X_{1},...,X_{d}$ be a fixed $\zp$-basis of $\H$. There is an isomorphism $\psi :\,\zp^{d} \ra H_{0}$ given by \begin{equation}\label{eqn:psi}
	\psi (t_{1},...,t_{2}) = \exp(t_{1}X_{1})...\exp(t_{d}X_{d}).
\end{equation} By the theory of Mahler expansions \cite{Lazard}, the space of $K$-valued continuous functions on $\z_{p}^{d}$, denoted by $C(\z_{p}^{d},K)$, is the space of all series of the form \[
	f(x)=\sum_{I\in \z_{\geq 0}^{d}} c_{I} \binom{t}{I} =\sum_{I\in \z_{\geq 0}^{d}} c_{I} \prod_{j=1}^{d} \binom{t_{j}}{i_{j}}
\] where $c_{I}\ra 0$ as $|I| \ra \infty$. Here, $I$ runs over all $d$-tuples of non-negative integers $(i_{1},...,i_{d})$ and $|I|=i_{1}+...+i_{d}$. By pulling back via the isomorphism $\psi$, there is an analogous description of the space of continuous functions on $H$, which we denote by $C(\h,K)$. For each $n\geq 0$, a theorem of Amice \cite[III.1.3.9]{Amice} shows that there is an isomorphism of $K$-Banach spaces \[ 
	C^{\an}(\h_{n},K)\cong\left\{\sum_{I\in \z_{\geq 0}^{d}} c_{I} \binom{t}{I}:\, |c_{I}|\cdot\left|\left[\frac{|I|}{p^{-n}}\right]!\right|^{-1}\ra 0 \textnormal{ as } |I|\ra \infty \right\}.
\] 

Let $D^{\an}(\h_{n},K)$ be the strong dual of $C^{\an}(\h_{n},K)$, which is also a $K$-Banach algebra. The space $C^{\an}(\h_{n},K)$ admits a continuous action of the universal enveloping algebra $U(\g)$ coming from differentiating the right regular action of $H_{n}$. This induces a continuous $K$-linear map $U(\g)\ra D^{\an}(\h_{n},K)$ defined by \[
    \left<\mathfrak{t},f\right> = (\mathfrak{t}f)(e)
\]
where $e$ is the identity element of the group $H_{n}$. Let $\alpha_{j}=\exp(X_{j})-1$. The theory of finite differences shows that the dual of $\binom{t_j}{i_{j}}$ is exactly $\alpha_{j}^{i_{j}}$. Thus we get an isomorphism of Banach $K$-algebras \begin{equation}\label{eqn:dist1}
    D^{\an}(\h_{n},K)\cong \left\{\sum_{I\in \z_{\geq 0}^{d}} b_{I} \alpha^{I}:\, |b_{I}|\leq L\cdot\left|\left[\frac{|I|}{p^{-n}}\right]!\right|^{-1} \textnormal{ for some constant } L>0 \right\},
\end{equation} where $\alpha^{(i_{1},...,i_{d})}$ is defined to be the product $\prod_{j=1}^{d}\alpha_{j}^{i_{j}}$.

Following \cite{analytic}, define $C^{\an}(\h^{\circ},K)$ to be the projective limit $\underset{\overset{\la}{n}}{\lim }\,C^{\an}\left(\h_{n},K\right)$, which is a nuclear Fr\'{e}chet space (see \cite[Def. 1.1]{summary}). We denote its strong dual by $D^{\an}(\h^{\circ},K)$, which is a coherent ring of compact type and is isomorphic to the direct limit $\underset{\overset{\ra}{n}}{\lim }\,D^{\an}\left(\h_{n},K\right)$ \cite[Cor. 5.3.12]{analytic}. The isomorphisms above show that \begin{equation}\label{eqn:dist2}
    D^{\an}(\h^{\circ}_{0},K)\cong \left\{\sum_{I\in \z_{\geq 0}^{d}} b_{I} \alpha^{I}:\, |b_{I}|\leq L\cdot\left|\left[\frac{|I|}{p^{-n}}\right]!\right|^{-1} \textnormal{ for some constant } L>0 \textnormal{ and } n\geq 0 \right\}.
\end{equation} Compare this to the descriptions found in \cite{analytic}, \cite[\textsection 4]{ST_dist} and \cite[\textsection 2]{ST_GL2}.

By our choice of good analytic open subgroups $\{H_{n}\}_{n\geq 0}$ of $G$, every $H_n$ is a normal subgroup of $H=H_0$. There is an isomorphism \[
    C^{\La}(H_{0},K)_{\h_{n}\An} \cong \underset{{g\in H_{0}/H_{n}}}{\bigoplus} C^{\an}(g\h_{n},K). 
\] Following the notation established in \cite{analytic}, we define \[
    C^{\La}(H_{0},K)_{\h^{\circ}_{n}\An} = \underset{\underset{m>n}{\la}}{\lim} C^{\La}(H_{0},K)_{\h_{m}\An} \cong \underset{\underset{m}{\la}}{\lim} \underset{{g\in H_{0}/H_{m}}}{\bigoplus} C^{\an}(g\h_{m},K),
\] which is a nuclear Fr\'echet space. Let $D(H_{0},\h^{\circ}_{n})$ be the strong dual of $C^{\La}(H_{0},K)_{\h^{\circ}_{n}\An}$. Then there is an isomorphism of topological $K$-algebras \begin{equation}\label{eqn:dist3}
    D(H_{0},\h^{\circ}_{n}) \cong \underset{\underset{m>n}{\ra}}{\lim}D(H_{0},\h_{n}) \cong \underset{\underset{m}{\ra}}{\lim} \underset{{g\in H_{0}/H_{m}}}{\bigoplus} \delta_{g}\star D^{\an}(\h_{m},K). 
\end{equation} This isomorphism shows that $D(H_{0},\h^{\circ}_{n})$ is a coherent ring of compact type.

For each $n\geq 0$, there is a continuous injection \[
    C^{\La}(H_{0},K)_{\h^{\circ}_{n}\An} \ra C^{\La}(H_{0},K)_{\h^{\circ}_{n+1}\An} 
\] which induces an isomorphism \[
    C^{\La}(H_{0},K) \cong \underset{\ra}{\lim}  C^{\La}(H_{0},K)_{\h^{\circ}_{n}\An}.
\] Let $D^{\La}(H_{0},K)$ be the strong dual of $C^{\La}(H_{0},K)$. There is an isomorphism of topological $K$-algebras \[
    D^{\La}(H_{0},K) \cong  \underset{\la}{\lim}  D(H_{0},\h^{\circ}_{n}).
\] It was shown in \cite[\textsection 5]{analytic} that this isomorphism induces a weak Fr\'echet-Stein algebra structure on $D^{\La}(H_{0},K)$, in the sense of \cite[Def. 1.2.6]{analytic}.

There are natural injections of $K[H_0]$ into each of the distribution algebras $D(H_{0},\h^{\circ}_{n})$ for all $n\geq 0$ and $D^{\La}(H_{0},K)$, by mapping $h\in H_{0}$ to the dirac distribution $\delta_{h}$. The image of both of these injections are dense. By Theorem 6.3 of \cite{ST_dist}, there is an anti-equivalence between the category of admissible $G$-representations over $K$ and left coadmissible $D^{\La}(H_{0},K)$-modules. The anti-equivalence is by sending an admissible $G$-representation $V$ to its strong dual $V^{\prime}_{b}$, where $h\in H_{0}$ acts on $V^{\prime}_{b}$ via the contragradient action associated to $h^{-1}$ (this is so that $V^{\prime}_{b}$ is a left $D^{\La}(H_{0},K)$-module). This action of $H_{0}$ on $V^{\prime}_{b}$ is compatible with the natural injection $K[H_{0}]\inj D^{\La}(H_{0},K)$ described earlier; that is, $\delta_{h}\star v^\prime=h^{-1}v^{\prime}$ for all $h\in H_0$ and $v^{\prime} \in V^{\prime}_{b}$. Taking $V=C^{\La}(H_{0},K)$, then the contragradient action of $h\in H_{0}$ on $D^{\La}(H_{0},K)$ (coming from the right regular action on $C^{\La}(H_{0},K)$) is the same as convolution on the left with $\delta_{h}$ (but note that $\delta_{h}\star \delta_{g} = \delta_{gh}$). 

\begin{lemma}\label{lem:Hd_coinv}
The space of $N_{0}$-coinvariants of $D(H_{0},\h^{\circ}_{n})$, denoted by $D(H_{0},\h^{\circ}_{n})_{N_{0}}$, coincides with the Hausdorff $N_{0}$-coinvariants. Additionally, there is an isomorphism of $D(M_{0},\m^{\circ}_{n})$-algebras \[
    D(H_{0},\h^{\circ}_{n})_{N_{0}} \cong D(\overline{P}_{0},\overline{\mathbb{P}}^{\circ}_{n})\cong D(\overline{N}_{0},\overline{\n}^{\circ}_{n})\widehat{\otimes}_{K} D(M_{0},\m^{\circ}_{n}), 
\] where $D(M_{0},\m^{\circ}_{n})$ acts on the target by its usual left module structure on the second factor. 
\end{lemma}
\begin{proof}
By the choice of the good analytic open subgroups $\{H_{n}\}_{n\geq 0}$, there is an isomorphism of rigid analytic groups\[
    \overline{\N}_{n} \oplus\M_{n}\oplus\N_{n}=\H_{n} \overset{\exp}{\cong} \h_{n}. 
\] Let $\{X_{j}\}_{j=1}^{\ell}$, $\{X_{j}\}_{j=\ell+1}^{d}$ be a choice of $\zp$-basis for $\overline{\N}_{n} \oplus\M_{n}$ and $\N_{n}$ respectively. Then their union forms a basis of $\H_{n}$. Define $\psi:\,\zp^{d} \cong H_{n}$ as in equation (\ref{eqn:psi}).


The action of $n\in N_{n}$ on $D^{\an}(\h^{\circ}_{n},K)$ is that of convolution with $\delta_{n}$. Suppose $n=\psi(\overline{t})$ is an element of $N_{n}$, for some $\overline{t}=(0,...,0,\overline{t}_{\ell+1},...,\overline{t}_{d})\in \z_{p}^{d}$. Then  \[
    \delta_{n}\left(\psi^{\star}(f)\right)= f(\bar{t}) 
\] and an easy computation shows that \[
    \delta_{n}\left(\psi^{\star}(t_{j}^{k_{j}})\right)=t_{j}^{k_j}(\bar{t})=\begin{cases}
    \bar{t}^{k_j}_{j} & \textnormal{if } \ell+1\leq j\leq d\\
    0 & \textnormal{else}\\
\end{cases}.
\] Since the dual of $t_{j}^{k_{j}}$ is $\frac{X_{j}^{k_{j}}}{k_{j}!}$, we see that under isomorphism (\ref{eqn:dist1}), $\delta_{n}$ is identified with the power series \[
    \sum_{I}\bar{t}^{k_{\ell+1}}_{\ell+1}...\bar{t}^{k_d}_{d}\frac{X_{\ell+1}^{k_{\ell+1}}}{k_{\ell+1}!}...\frac{X_{d}^{k_{d}}}{k_{d}!}.
\] In particular, for all $\ell+1\leq j \leq d$, the elements $\delta_{\exp(X_{j})}-1$ correspond to $\exp(X_{j})-1=\alpha_{j}$. Equations (\ref{eqn:dist1}) and (\ref{eqn:dist2}) combine to show that $D^{\an}(\h_{n}^{\circ},K)_{N_{n}}$ can be identified with the subalgebra generated by $\alpha_{j}$ for all $j\leq\ell$. This is isomorphic to $D^{\an}(\overline{\mathbb{P}}_{n}^{\circ},K)$. Equation (\ref{eqn:dist3}) implies that \begin{equation*}
        D(H_{0},\h^{\circ}_{n})_{N_{0}} \cong D(\overline{P}_{0},\overline{\mathbb{P}}^{\circ}_{n}) 
\end{equation*} which is Hausdorff. 

The second assertion of the lemma follows from the Iwahori decomposition of $H_{0}$ and $H_{n}$ and the proof of \cite[Prop. 4.2.22]{Jacquet1} using the untwisting Lemma \cite[Lem. 3.6.4]{analytic}. 
\end{proof}

Suppose $H$ is an analytic open subgroup of $G$ obtained by exponentiating a $\zp$-Lie sublattice $\H$ of $\g$, where $\left[\H,\H\right]\sub a\H$ for some $a^{p-1}\in p\zp$. Suppose $L$ is an analytic open subgroup of $H$ of finite index, obtained by exponentiating a $\zp$-Lie sublattice $\mathfrak{l}\sub\H$. By the theory of elementary divisors, there exists a $\zp$-basis $\{X_{j}\}$ of $\H$ and elements $\alpha_{j}\in \zp$, such that $\left\{ \alpha_{j}X_{j}\right\}$  is a basis of $\mathfrak{l}$. 

Let $A^{(m)}$ (resp. $C^{(m)}$) be the $\ok$-subalgebra of $U(\g)$ generated by $\frac{X_{j}^{i}}{i!}$ (resp. $\frac{(\alpha_{j}X_{j})^{i}}{i!}$) for all $0\leq i\leq p^{m}$ and $0\leq j \leq d$. More explicitly,  \begin{equation}\label{eqn:A^(m)}
    \begin{split}
        A^{(m)}& = \left\{ \sum_{I} a_{I} \frac{q(i_{1})!...q(i_{d})!}{i_{1}!...i_{d}!}X_{1}^{i_{1}}...X_{d}^{i_{d}}: a_{I}\in \ok \textnormal{ and } a_{I}=0\textnormal{ for almost all }I\right\} \textnormal{ and}\\
        C^{(m)}& = \left\{ \sum_{I} b_{I} \frac{q(i_{1})!...q(i_{d})!}{i_{1}!...i_{d}!}(\alpha_{1}X_{1})^{i_{1}}...(\alpha_{d}X_{d})^{i_{d}}: b_{I}\in \ok \textnormal{ and } b_{I}=0\textnormal{ for almost all }I\right\},
    \end{split}
\end{equation} where $q(i_j)$ is the integral part of $\frac{i_{j}}{p^m}$. Let $\widehat{A}^{(m)}$ (resp. $\widehat{C}^{(m)}$) denote the $p$-adic completion of $A^{(m)}$ (resp. $C^{(m)}$). Let $D^{\an}(\h^{\circ},K)^{(m)}= K\otimes_{\ok}\widehat{A}^{(m)}$ (resp. $D^{\an}(\L^{\circ},K)^{(m)}= K\otimes_{\ok}\widehat{C}^{(m)}$), which is naturally a $K$-Banach algebra. There is an isomorphism of compact type $K$-algebras \begin{equation}\label{eqn: D^(m)}
    D^{\an}(\h^{\circ},K)\cong \underset{\underset{m}{\ra}}{\lim}D^{\an}(\h^{\circ},K)^{(m)} \textnormal{ (resp. } D^{\an}(\L^{\circ},K)\cong \underset{\underset{m}{\ra}}{\lim}D^{\an}(\L^{\circ},K)^{(m)} \textnormal{)}
\end{equation} whose transition maps are flat and dense \cite[Prop. 5.2.6, Prop. 5.3.11]{analytic}. Let $A_{\leq i}^{(m)}=A^{(m)}\cap U(\g)_{\leq i}$ (resp. $C_{\leq i}^{(m)}=C^{(m)}\cap U(\g)_{\leq i}$). 

\begin{proposition}\label{prop:commutator}
Suppose $\left[\H,\H\right]\sub a\H$ for some $a^{p-1}\in p\zp$ and $m,m^{\prime}\geq 0$, then \[
    \left[C^{(m^{\prime})},A^{(m)}_{\leq i}\right] \sub C^{(m^{\prime})}A^{(m)}_{\leq i-1}
\] for all $i>0$. 
\end{proposition}
\begin{proof}
This is a slight generalization of \cite[Prop. 5.2.17]{analytic}, where the only difference is that all the $\alpha_{j}$'s are all equal. The same proof applies to the case where the $\alpha_{j}$'s may differ from one another.   
\end{proof}

\begin{theorem}\label{thm:flat}
Let $\h$ and $\L$ be rigid analytic groups satisfying all the hypothesis described above. If $m$ is a sufficiently large integer, then the natural map \[
    D^{\an}\left(\L^{\circ},K\right)\ra D^{\an}\left(\h^{\circ},K\right)\]
induced by the rigid analytic inclusion $\L\sub \h$ factors through the natural map \[
    D^{\an}\left(\h^{\circ},K\right)^{(m)}\ra D^{\an}\left(\h^{\circ},K\right)
\] and the resulting map \[
    D^{\an}\left(\L^{\circ},K\right)\ra D^{\an}\left(\h^{\circ},K\right)^{(m)}
\] is flat.
\end{theorem}
\begin{proof}
The proof is similar to that of \cite[Prop. 5.3.13]{analytic}. We would like point out that there is a small error in the proof of \cite[Prop. 5.3.13]{analytic} (the equation after (5.3.15)), which can be fixed in a similar way to the proof given here. 

For each natural number $i$, let $s(i)$ be the sum of digits in the $p$-adic expansion of $i$. Let $i$ be a fixed integer, and write $i=p^{m}q+r$ (resp. $i=p^{m^\prime}q^\prime +r^\prime$) for some $q^\prime$ and $0\leq r<p^{m}$ (resp. $q^\prime$ and $0\leq r^\prime < p^{m^\prime})$. Then \[
    \ord_{K}(q!)=\ord_{K}\left(\frac{i-r}{p^{m}}!\right)= \frac{i-r}{p^{m}(p-1)}-\frac{s(i-r)}{p-1}.
\]
Choose $m$ large enough so that \[
    \ord_{K}(\alpha_{j})\geq\frac{1}{p^{m}(p-1)}
\] for all $j$. The inequality $\ord_{K}(\alpha_{j})i\geq\ord_{K}(q!)$ implies that \[
    \ord_{K}\left(\alpha_{j}^{i}q^{\prime}!\right)=\ord_{K}(\alpha_{j})i+\ord_{K}(q^{\prime}!)\geq\ord_{K}(q!)
\]
independently of $m^{\prime}$. Taking into account of equation (\ref{eqn:A^(m)}), this shows that $C^{(m^{\prime})}\sub A^{(m)}$ for every $m^{\prime}$. After taking $p$-adic completion and tensoring with $K$, we see that there is a factorization \[
    D(\L^{\circ},K)^{(m^{\prime})}\ra D^{\an}\left(\h^{\circ},K\right)^{(m)}\ra D^{\an}\left(\h^{\circ},K\right)
\] for every $m^{\prime}$. Therefore, there is a factorization \[
    D(\L^{\circ},K)\ra D^{\an}\left(\h^{\circ},K\right)^{(m)}\ra D^{\an}\left(\h^{\circ},K\right).
\]
Let $B=C^{(m^{\prime})}A^{(m)}$, which by Proposition \ref{prop:commutator} is an $\ok$-subalgebra of $U(\g)$. In fact, since $C^{(m^{\prime})}\sub A^{(m)}$, $B=A^{(m)}$. Equip $B$ with the filtration $F_{i}B=C^{(m^{\prime})}A_{\leq i}^{(m)}$. By Proposition \ref{prop:commutator} again, $B$ satisfies the all the assumptions of \cite[Lem. 5.3.9]{analytic} (by taking $A=C^{(m^{\prime})}$). Proposition 5.3.10 of \cite{analytic} shows that \[
    D(\L^{\circ},K)^{(m^{\prime})}=K\otimes_{\ok}\widehat{C}^{(m^{\prime})}\ra K\otimes_{\ok}\widehat{B}\cong K\otimes_{\ok}\widehat{A}^{(m)}= D^{an}\left(\h^{\circ},K\right)^{(m)}
\] is flat as required.
\end{proof}

Suppose additionally that $\{H_{n}\}_{n\geq 0}$ is sequences of good analytic open subgroups of $G$ satisfying proposition \ref{prop:good_grps}, where $H=H_{0}$. For each $n\geq 0$, let $L_{n}=H_{n}\cap L$. The inclusions $L_{0}\sub H_{0}$ and $L_{n}\sub H_{n}$ induce a continuous homomorphism of Fr\'{e}chet $K$-algebras \[
    D^{\La}(L_{0},K) \ra D^{\La}(H_{0},K),
\] as well as a continuous homomorphism of compact type $K$-algebras \[
    D(L_{0},\L_{n}^{\circ}) \ra D(H_{0},\h_{n}^{\circ}) .
\] Furthermore, there is an isomorphism of topological $K$-algebras \begin{equation*}
    D^{\La}(H_{0},K) \cong \underset{x\in H_{0}/L_{0}}{\bigoplus} \delta_{x}\star D^{\La}(L_{0},K).
\end{equation*} For each $n\geq 0$, define \begin{align*}
    D(H_{0},\L_{n}^{\circ}) & := D(L_{0},\L_{n}^{\circ}) \underset{D^{\La}(L_{0},K)}{\otimes} D^{\La}(H_{0},K) \\
    & \cong \underset{x\in H_{0}/L_{0}}{\bigoplus} \delta_{x}\star D(L_{0},\L_{n}^{\circ}).
\end{align*} 

For all $n$ sufficiently large, there are also inclusions $H_{n}\sub L_{0}$. Since $H_{n}$ is a normal subgroup of the bigger group $H_{0}$, $H_{n}$ is also a normal subgroup of $L_{0}$. Define the compact type $K$-algebra $D(L_{0},\h_{n}^{\circ})$  to be the strong dual of $C^{\La}(L_{0},K)_{\h_{n}^{\circ}\An}$. Similar to equation (\ref{eqn:dist3}), there is an isomorphic of $K$-algebras \[
    D(L_{0},\h_{n}^{\circ}) \cong \underset{\underset{m>n}{\ra}}{\lim} \underset{{g\in L_{0}/H_{m}}}{\bigoplus} \delta_{g}\star D^{\an}(\h_{m},K). 
\]

\begin{corollary}\label{cor:flat}
Let $n$ be an integer large enough so that $H_{n}\sub L_{0}$. Then the continuous $K$-linear morphism between compact type $K$-algebras \begin{equation*}
    D(L_{0},\L_{n}^{\circ}) \ra D(L_{0},\h_{n}^{\circ}) 
\end{equation*}
is flat.
\end{corollary}
\begin{proof} By dualizing the natural map $C^{\La}(H_{n}^{\circ},K)_{\h_{n}^{\circ}\An}\ra C^{\La}(H_{n}^{\circ},K)_{\L_{n}^{\circ}\An}$, we obtain a map \begin{equation}\label{eqn:D(H_n,L_n)}
    D(H_{n}^{\circ},\L_{n}^{\circ}) \ra D^{\an}(\h_{n}^{\circ},K). 
\end{equation} Since $D(H_{n}^{\circ},\L_{n}^{\circ})$ is a free $D^{\an}(\L_{n}^{\circ},K)$-module of finite rank, by Theorem \ref{thm:flat}, the above morphism factors through some $D^{\an}(\h_{n}^{\circ},K)^{(m)}$. 

By the same argument as \cite[Prop. 5.3.18]{analytic}, the induced map $D(H_{n}^{\circ},\L_{n}^{\circ})\ra D^{\an}(\h_{n}^{\circ},K)^{(m)}$ is flat. By \cite[Rem. 3.2]{ST_dist} applied to (\ref{eqn: D^(m)}), the map $D^{\an}(\h_{n}^{\circ},K)^{(m)} \ra  D^{\an}(\h_{n}^{\circ},K)$ is also flat. Therefore, equation ($\ref{eqn:D(H_n,L_n)}$) is flat. The corollary then follows from tensoring equation ($\ref{eqn:D(H_n,L_n)}$) by $D^{\La}(L_{0},K)$ over $D^{\La}(H_{n}^{\circ},K)$. 
\end{proof}


\vspace{3mm}
\section{Derived Jacquet-Emerton Module}\label{sec:derived}

\vspace{3mm}

\subsection{Jacquet-Emerton Module Functor}

In this subsection, we recall the construction of the Jacquet-Emerton functor found in \cite{Jacquet1} and prove a theorem involving the exactness of the functor $I_{\overline{P}}^{G}$ introduced in \cite{Jacquet2}. 


Keep all the notation established in Section \ref{sec:intro} and \ref{sec:dist}. Following the notations established in \cite{Jacquet1}, let $\rep_{\textnormal{top.c}}(G)$ denote the category of topological $G$-representations on Hausdorff locally convex $K$-vector spaces of compact type. The morphisms are continuous $G$-equivariant $K$-linear maps. Let $\rep_{\textnormal{la.c}}(G)$ be the full subcategorry of $\rep_{\textnormal{top.c}}(G)$ consisting of locally analytic representations of $G$. Let $\rep_{\textnormal{es}}$ (resp. $\rep_{\textnormal{ad}}$) be the full subcategory of $\rep_{\textnormal{la.c}}(G)$ consisting of essentially admissible (resp. admissible) locally analytic $G$-representations over $K$. Let $\rep_{\textnormal{la.c}}^{z}(G)$ be the full subcategory of $\rep_{\textnormal{la.c}}(G)$ consisting of objects that satisfy the equivalent conditions of \cite[Prop. 6.4.7]{analytic}.


Let $\widehat{Z}_{M}$ be the rigid analytic space of locally $\qp$-analytic characters on $Z_{M}$. Suppose $V$ is an object of $\rep_{\textnormal{top.c}}(Z_{M}^{+})$. Define the finite slope part of $V$ to be \[
    V_{\fs} = \l_{b,Z_{M}^{+}}(C^{\an}(\widehat{Z}_{M},K),V).
\] According to \cite[Lem. 3.2.3]{Jacquet1}, there is a natural isomorphism \begin{equation}\label{eqn:fs_dual}
    (V_{\fs})^{\prime}_{b}\cong C^{\an}(\widehat{Z}_{M},K)\underset{K[Z_{M}^{+}]}{\widehat{\otimes}} V^{\prime}_{b}.
\end{equation}



Let $\delta_{P}: P\ra K^{\times}$ be the modulus character of the parabolic subgroup $P$. This is a smooth character that factors through the Levi quotient $M$. Explicitly, \[
    \delta_{P}(m)=\frac{1}{\left[N_{0}:mN_{0}m^{-1}\right]}
\] for all $m\in M$ and any compact open subgroup $N_{0}$ of $N$. 
Suppose $V$ is a locally $\qp$-analytic representation of $P$ on a locally convex $K$-vector space. For each $z\in Z_{M}^{+}$, define the Hecke operator $\pi_{N_{0},z}:\,V^{N_{0}} \ra V^{N_{0}}$ to be the endomorphism \begin{equation}\label{eqn:original_Hecke}
      \pi_{N_{0},z}(v)=\delta_{P}(z) \underset{x\in N_{0}/zN_{0}z^{-1}}{\sum} xzv = \frac{1}{\left[N_{0}:zN_{0}z^{-1}\right]} \underset{x\in N_{0}/zN_{0}z^{-1}}{\sum} xzv.
\end{equation}
\begin{definition}\label{def:Jacquet}
The Jacquet-Emerton module functor $J_{P}:\,\rep_{\textnormal{es}}(G)\ra \rep_{\textnormal{es}}(M)$ is defined to be \[
    V \map (V^{N_{0}})_{\textnormal{fs}} = \l_{b,Z_{M}^{+}}(C^{\an}(\widehat{Z}_{M},K),V^{N_{0}}).
\]
\end{definition}

We end this subsection with a theorem that will be useful for a computation in section \ref{sec:applications}. Let $U$ be an object of $\rep_{\textnormal{la,c}}^{z}(M)$. Following \cite{Jacquet2}, we define $I_{\overline{P}}^{G}(U)$ to be the closed $G$-subrepresentation of $\ind_{\overline{P}}^{G}(U)$ generated by the image of $U(\delta_{P})$ under the canonical lift $J_{P}(\ind_{\overline{P}}^{G}U)\ra \ind_{\overline{P}}^{G}U$ defined in \cite[(3.4.8)]{Jacquet1}. For more properties of $I_{\overline{P}}^{G}(U)$, see \cite[\textsection 2]{Jacquet2}. 

\begin{theorem}\label{thm:exact_IpG}
Suppose $0 \ra U \ra V \ra W \ra 0$ is a strict exact sequence in $\rep_{\textnormal{la,c}}^{z}(M)$. Suppose $U(\g)\otimes_{U(\p)}U$ is irreducible. Then \begin{equation*}
    0 \ra I_{\overline{P}}^{G}(U) \ra I_{\overline{P}}^{G}(V) \ra I_{\overline{P}}^{G}(W) \ra 0
\end{equation*} is a strict exact sequence in $\rep_{\textnormal{la,c}}^{z}(G)$. 
\end{theorem}
\begin{proof}
By Corollary 4.14 of \cite{Schraen}, the following is a strict short exact sequence of locally analytic $G$-representations \begin{equation}\label{eqn:ind_exact}
    0 \ra \ind_{\overline{P}}^{G} U \ra \ind_{\overline{P}}^{G}V \ra \ind_{\overline{P}}^{G} W \ra 0.
\end{equation}
Since $\ind_{\overline{P}}^{G} U \ra \ind_{\overline{P}}^{G}V$ is a closed embedding, it follows immediately that $I_{\overline{P}}^{G}(U) \ra I_{\overline{P}}^{G}(V)$ is also a closed embedding. On the other hand, there is a commutative diagram \begin{equation}
\begin{tikzcd}
    V(\delta_{P}) \arrow[r,two heads] \arrow[d] & W(\delta_{P}) \arrow[d] \\
	\ind_{\overline{P}}^{G}V \arrow[r, two heads] & \ind_{\overline{P}}^{G}W
\end{tikzcd}.
\end{equation}
The image of the left vertical arrow surjects onto the image of the second vertical arrow. Since the bottom horizontal arrow is $G$-equivariant and strict, we have just shown that there is a $G$-equivariant strict surjection $I_{\overline{P}}^{G}(V) \ra I_{\overline{P}}^{G}(W)$. It now suffices to check exactness in the middle. 

The kernel of $I_{\overline{P}}^{G}(V) \ra I_{\overline{P}}^{G}(W)$ is exactly the intersection \begin{equation}
    X := I_{\overline{P}}^{G}(V) \cap \ind_{\overline{P}}^{G} U
\end{equation} inside $\ind_{\overline{P}}^{G} V$. It suffices to show that this is exactly $I_{\overline{P}}^{G}(U)$. The space $X$ is a local closed subrepresentation of $\ind_{\overline{P}}^{G} V$ in the sense of \cite[Def. 2.4.1]{Jacquet2}. Let $e$ be the image of the identity of $G$ in $\overline{P}\backslash G$. The remark following \cite[Prop. 2.4.9]{Jacquet2} implies that there is an one to one correspondence between local closed $G$-invariant subspaces of $\ind_{\overline{P}}^{G} V$ and the closed $(\g,\overline{P})$-invariant subspaces of $(\ind_{\overline{P}}^{G} V)_{e}$, by mapping a local closed subspace to its stalk at $e$, as defined by \cite[Def. 2.4.2]{Jacquet2}. Our goal is to show that there is an equality on the stalks $X_{e}=I_{\overline{P}}^{G}(U)_{e}$. Since both of these spaces are polynomially generated, in the sense of \cite[Def. 2.7.15]{Jacquet2}, it also suffices to show that $X_{e}^{\pol}=I_{\overline{P}}^{G}(U)_{e}^{\pol}$. 

From the proof of Proposition $2.8.10$ in \cite{Jacquet2}, the space $I_{\overline{P}}^{G}(U)_{e}^{\pol} $ is exactly the image of $U(\g)\otimes_{U(\p)}U\ra C^{\pol}(N,U)$. By \cite[Prop. 3.26]{Bergall_Chojecki}, the spaces $U(\g)\otimes_{U(\p)}U$ and $C^{\pol}(N,U)$ are in fact isomorphic. Meanwhile, \begin{equation}\begin{split}
    X_{e}^{\pol}	& = I_{\overline{P}}^{G}(V)_{e}^{\pol}\cap(\ind_{\overline{P}}^{G}U)_{e}^{\pol} \\
	& =	I_{\overline{P}}^{G}(V)_{e}^{\pol}\cap C^{\pol}(N,U) \\
	&=	\im\left(U(\g)\otimes_{U(\p)}V\ra C^{\pol}(N,V)\right)\cap C^{\pol}(N,U),
\end{split}\end{equation} which is exactly $I_{\overline{P}}^{G}(U)_{e}^{\pol}$ as required. 
\end{proof}

\begin{remark}
The assumption that $G$ is a locally $\qp$-analytic group is only being used to show that (\ref{eqn:ind_exact}) is a strict short exact sequence. Suppose $L$ is a finite extension of $\qp$, and $G$ is a connected locally $L$-analytic reductive group. Then the theorem still holds with the additional assumption that (\ref{eqn:ind_exact}) is a strict short exact sequence. 
\end{remark}

\subsection{Homology groups}

For the remainder of this section, fix an admissible locally analytic $G$-representation $V\in \rep_{\textnormal{ad}}(G)$. Fix a decreasing sequence of good analytic open subgroups $\{H_{n}\}_{n\geq 0}$ of $G$ satisfying Proposition \ref{prop:good_grps}. For any ring $R$, let $Coh_{R}$ denote the abelian category of coherent $R$-modules \cite[Thm. 2.3]{Swan}. If $R$ is a coherent ring, then $Coh_{R}$ coincides with the category of finitely presented $R$-modules \cite[Cor. 2.7]{Swan}.

It is shown in \cite[Prop. 5.3.1]{analytic} that the isomorphism of topological $K$-algebras \[
    D^{\La}(H_{0},K)\cong \underset{\la}{\lim}\dn
\] induces a weak Fr\'{e}chet-Stein algebra structure on $D^{\La}(H_{0},K)$. By Theorem 1.2.11 of \cite{analytic}, $\vn$ is a finitely generated module of the coherent ring $D(H_{0},\h_{n}^{\circ})$, for each $n\geq 0$. Small augmentations to the proof of Proposition $A.1$ and Lemma $A.11$ of \cite{Jacquet2} show that there are topological isomorphisms \begin{equation} \label{eqn:D(Hn)V}
    \vn\cong \dn\widehat{\otimes}_{D^{\La}(H_{0},K)}V^{\prime}_{b} \cong \dn\otimes_{D^{\La}(H_{0},K)}V^{\prime}_{b}
\end{equation}
and $\vn$ is finitely presented as a module over the coherent ring $\dn$ for all $n\geq 0$. 

\begin{lemma}\label{lem:free_res}
Suppose $M$ is a finitely presented module of a coherent ring $R$. Then there is a projective resolution of $M$ consisting of free $R$-modules of finite rank. 
\end{lemma}
\begin{proof}
This is \cite[Cor. 2.5.2]{Glaz} or \cite[Rem. 1.4]{Gersten}.
\end{proof}

Let  $P_{n,\star}\ra \vn$ be a projective resolution of $\vn$ in the category $Coh_{D(H_{0},\h_{n}^{\circ})}$. The existence of such a resolution is guaranteed by the previous lemma. Proposition $A.10$ of \cite{Jacquet2} shows that each term $P_{n,\star}$ has a unique topology making them compact type $\dn$-modules. Additionally, the boundary maps are automatically continuous and strict. For each $n\geq 0$, let \[
    H_{\star}(N_{0},(V_{\h_{n}^{\circ}\An})^{\prime}_{b}) = H_{\star}((P_{n,\star})_{N_{0}})
\] be the usual group homology of $N_{0}$-coinvariants. It is important to note that by Lemma \ref{lem:Hd_coinv}, the Hausdorff $N_{0}$-coinvariants of each $P_{n,\star}$ coincides with the usual $N_{0}$-coinvariants, and so each term $(P_{n,\star})_{N_{0}}$ of the complex is a compact type $K$-algebra. For each $n,k\geq0$, let $\d_{n,k}:\,(P_{n,k})_{N_{0}}\ra (P_{n,k-1})_{N_{0}}$ be the boundary operators. Let $\widehat{H}_{\star}(N_{0},(V_{\h_{n}^{\circ}\An})^{\prime}_{b})$ be the Hausdorff completion of $H_{\star}(N_{0},(V_{\h_{n}^{\circ}\An})^{\prime}_{b})$. It can also be identified with the quotient of $\ker(\d_{k})$ by the closure of $\im(\d_{k+1})$ and is a $K$-algebra of compact type. Let $H^{\star}(N_{0},V_{\h_{n}^{\circ}\An})$ denote the strong dual of $\widehat{H}_{\star}(N_{0},(V_{\h_{n}^{\circ}\An})^{\prime}_{b})$, which is a nuclear Fr\'echet space.

Suppose $Q_{n,\star}\ra \vn$ is another projective resolution in the category $Coh_{D(H_{0},\h_{n}^{\circ})}$. The identity map on $\vn$ induces maps between complexes $P_{n,\star} \ra Q_{n,\star}$ and $Q_{n,\star}\ra P_{n,\star}$, which induces isomorphisms on the homology groups. Each of these maps are automatically continuous and strict \cite[Prop. A.10]{Jacquet2}. Therefore, the induced topology on the homology groups $H_{\star}(N_{0},\vn)$ does not depend on the choice of resolutions. 

\begin{proposition}\label{prop:H(V_L)->H(V_H)}
Suppose $L_{0}\sub H_{0}$ is a good analytic open subgroup. For each $n\geq 0$, let $L_{n}=H_{n}\cap L_{0}$, $M_{n}^{\prime} = M\cap L_{n}$, $\overline{P}_{n}^{\prime} = \overline{P}\cap L_{n}$, $N_{n}^{\prime}= N\cap L_{n}$ and let $\L_{n}$, $\m_{n}^{\prime}$, $\overline{\P}_{n}^{\prime}$ and $\n^{\prime}_{n}$ be their respective rigid analytic Zariski closures in $\L_{n}$. Let $m,n$ be integers large enough so that $H_{n}\sub L_{0}$ and $L_{m}\sub H_{n}$. Then there is a natural continuous $K$-linear map \begin{equation}\label{eqn:H(V_L)->H(V_H)}
    H_{\star}(N_{0}^{\prime},(V_{\L_{m}^{\circ}\An})^{\prime}_{b}) \ra H_{\star}(N_{0},(V_{\h_{n}^{\circ}\An})^{\prime}_{b})
\end{equation} which is induced by the natural continuous map   $((V_{\L_{m}^{\circ}\An})^{\prime}_{b})_{N^{\prime}_{0}} \ra ((V_{\h_{n}^{\circ}\An})^{\prime}_{b})_{N_{0}}$. 
In addition, there is an isomorphism between topological $D(M_{0}^{\prime},\m_{n}^{\prime,\circ})$-modules \begin{equation}\label{eqn:isom_of_}
    D(\overline{P}^{\prime}_{0},\overline{\P}_{n}^{\circ})\underset{D(\overline{P}^{\prime}_{0},\overline{\P}^{\prime,\circ}_{m})}{\otimes} H_{\star}(N_{0}^{\prime},(V_{\L_{m}^{\circ}\An})^{\prime}_{b}) \cong H_{\star}(N_{0}^{\prime},(V_{\h_{n}^{\circ}\An})^{\prime}_{b}).
\end{equation}
\end{proposition} 
\begin{proof}
Since \begin{equation}
    \dn = \underset{x\in H_{0}/L_{0}}{\bigoplus} \delta_{x} \star D(L_{0},\h_{n}^{\circ})
\end{equation}
is a free $D(L_{0},\h_{n}^{\circ})$-module of finite rank, finitely presented projective $\dn$-modules are also finitely presented projective $D(L_{0},\h_{n}^{\circ})$-modules. This means that the homology group $H_{\star}(N_{0},\vn)$ can be computed by viewing $\vn$ as a $\dn$-module or as a $D(L_{0},\h_{n}^{\circ})$-module. 

Let $Q_{\star}\ra (V_{\L_{m}^{\circ}\An})^{\prime}_{b}$ be a projective resolution of $(V_{\L_{m}^{\circ}\An})^{\prime}_{b}$ by finite free $D(L_{0},\L_{m}^{\circ})$-modules. Similar to equation (\ref{eqn:D(Hn)V}), there are isomorphisms \begin{equation}\label{eqn:vlm_vn}\begin{split}
    \vn & \cong D(L_{0},\h_{n}^{\circ})\otimes_{D^{\La}(L_{0},K)}V^{\prime}_{b} \textnormal{ and} \\
    (V_{\L_{m}^{\circ}\An})^{\prime}_{b} & \cong D(L_{0},\L_{m}^{\circ})\otimes_{D^{\La}(L_{0},K)}V^{\prime}_{b}.
\end{split}\end{equation}
Let $P_{\star}$ be the base change of $Q_{\star}$ to $D(L_{0},\h_{n}^{\circ})$ over $D(L_{0},\L_{m}^{\circ})$. By Corollary \ref{cor:flat} and equation (\ref{eqn:vlm_vn}), $P_{\star}\ra \vn$ is a projective resolution of $\vn$ by finite free $D(L_{0},\h_{n}^{\circ})$-modules. By construction, there is a natural map of complexes $(Q_{\star})_{N_{0}^{\prime}}\ra (P_{\star})_{N_{0}}$ and the first assertion of the proposition follows.

We can compute $H_{\star}(N_{0}^{\prime},(V_{\h_{n}^{\circ}\An})^{\prime}_{b})$ by taking the $N_{0}^{\prime}$-coinvariants from the resolution $P_{\star}\ra \vn$. By the same argument as before, $H_{\star}(N_{0}^{\prime},(V_{\h_{n}^{\circ}\An})^{\prime}_{b})$ is naturally a $K$-vector space of compact type. It follows from the flatness of $D(\overline{P}^{\prime}_{0},\overline{\P}^{\prime,\circ}_{m}) \ra D(\overline{P}^{\prime}_{0},\overline{\P}_{n}^{\circ})$ (Corollary \ref{cor:flat}) that there is an isomorphism of topological $D(M_{0}^{\prime},\m_{n}^{\circ})$-modules \begin{equation*}\begin{split}
    D(\overline{P}^{\prime}_{0},\overline{\P}_{n}^{\circ})\underset{D(\overline{P}^{\prime}_{0},\overline{\P}^{\prime,\circ}_{m})}{\otimes} H_{\star}((Q_{\star})_{N_{0}^{\prime}}) \cong H_{\star}\left( D(\overline{P}^{\prime}_{0},\overline{\P}_{n}^{\circ})\underset{D(\overline{P}^{\prime}_{0},\overline{\P}^{\prime,\circ}_{m})}{\otimes} (Q_{\star})_{N_{0}^{\prime}}\right) \cong H_{\star}((P_{\star})_{N_{0}^{\prime}}).
\end{split} 
\end{equation*} This completes the proof of the proposition.   
\end{proof}

\begin{remark}
In the special case where $\{L_{n}\}_{n\geq 0}=\{H_{n}\}_{n\geq 0}$ and $m=n+1$, we have defined a natural transition map \[
    H_{\star}(N_{0},(V_{\h_{n+1}^{\circ}\An})^{\prime}_{b}) \ra H_{\star}(N_{0},(V_{\h_{n}^{\circ}\An})^{\prime}_{b})
\] which is additionally $D(M_{0},\m_{n+1}^{\circ})$-linear, because the natural map $D(H_{0},\h_{n+1}^{\circ})_{N_{0}}\ra D(H_{0},\h_{n}^{\circ})_{N_{0}}$ is $D(M_{0},\m_{n+1}^{\circ})$-linear. 
\end{remark}

\begin{corollary}\label{cor:same_inv_limit}
Suppose $\{H_{n}\}_{n\geq 0}$ and $\{\widetilde{H}_{n}\}_{n\geq 0}$ are two sequences of good analytic open subgroups of $G$ satisfying Proposition \ref{prop:good_grps}. Assume that $H_{0}\cap N = N_{0} = \widetilde{H}_{0}\cap N$. Let $M^{\prime}_{0}=H_{0}\cap \widetilde{H}_{0} \cap M$. Then there is a natural isomorphism of topological $D^{\La}(M^{\prime}_{0},K)$-modules \begin{equation}
    \underset{\underset{n}{\la}}{\lim}\widehat{H}_{\star}(N_{0},(V_{\h_{n}^{\circ}\An})^{\prime}_{b}) \cong \underset{\underset{n}{\la}}{\lim}\widehat{H}_{\star}(N_{0},(V_{\widetilde{\h}^{\circ}_{n}\An})^{\prime}_{b}).
\end{equation} 
\end{corollary}
\begin{proof}
Let $L_{0}= H_{0}\cap \widetilde{H}_{0}$, and define $L_{n}$, $M_{n}^{\prime}$, $\overline{P}_{n}^{\prime}$, $N_{n}^{\prime}$, $\L_{n}$, $\n_{n}^{\prime}$, $\overline{\P}_{n}^{\prime}$ and $\n^{\prime}_{n}$ as in the statement of the previous proposition. By taking the Hausdorff completion followed by projective limit of (\ref{eqn:isom_of_}), there is an isomorphism of topological $D^{\La}(M^{\prime}_{0},K)$-modules \begin{equation*}
    \underset{\underset{n}{\la}}{\lim}\widehat{H}_{\star}(N_{0},(V_{\h_{n}^{\circ}\An})^{\prime}_{b}) \cong \underset{\underset{n}{\la}}{\lim}\widehat{H}_{\star}(N_{0},(V_{\L^{\circ}_{n}\An})^{\prime}_{b}).
\end{equation*} The corollary follows from composing with the inverse of the analogous isomorphism derived using $\widetilde{H}_{0}$ in place of $H_{0}$. 
\end{proof}

We close this subsection with an important theorem that is critical for computing these homology groups. The reader can compare the similarity of this result to \cite[Thm. 4.10, Thm. 7.1]{Kohlhaase}. The author would like to thank Vaughan McDonald for pointing out a mistake relating to the theorem in an earlier draft. 

\begin{theorem}\label{thm:lie_cohom}
Let $\N$ denote the Lie algebra of $N$ and suppose $V$ is an admissible locally analytic $G$ representation. Then there is a natural isomorphism \begin{equation}
    \underset{\underset{n}{\la}}{\lim} \,H_{\star}(N_{0},\vn) \cong H_{\star}(\N,V^\prime_{b})_{N_{0}}. 
\end{equation}
\end{theorem}
\begin{proof}
Let $C^{\sm}(N_{0},K)$ be the space of locally constant functions on $N_{0}$. Let $D^{\sm}(N_{0},K)$ and $D^{\sm}(N_{0},\n_{n}^{\circ})$ be the strong duals of $C^{\sm}(N_{0},K)$ and $C^{\sm}(N_{0},K)_{\n_{n}^{\circ}\An}$ respectively. These can be identified with the $\N$-coinvariants of $D^{\La}(N_{0},K)$ and $D(N_{0},\n_{n}^{\circ})$ respectively. There is an isomorphism \begin{equation*}\begin{split}
    D(N_{0},\n_{n}^{\circ})\otimes_{D^{\La}(N_{0},K)}D^{\sm}(N_{0},K) & \cong D(N_{0},\n_{n}^{\circ})\otimes_{D^{\La}(N_{0},K)}D^{\La}(N_{0},K) / D^{\La}(N_{0},K)\N \\
    & \cong D(N_{0},\n_{n}^{\circ})/D(N_{0},\n_{n}^{\circ})\N \cong D^{\sm}(N_{0},\n_{n}^{\circ}).
\end{split}
\end{equation*}

From \cite[(47)]{Kohlhaase} and \cite[pg. 306]{ST_duality}, there is an isomorphism \begin{equation}
    \Tor_{\star}^{D^{\La}(N_{0},K)}(\vn, D^{\sm}(N_{0},K)) \cong H_{\star}(\N,\vn),
\end{equation} where the torsion group is computed in the category of abstract $D^{\La}(N_{0},K)$-modules. This is further isomorphic to \begin{equation}\label{eqn:tor}\begin{split}
    \Tor_{\star}^{D(N_{0},\n_{n}^{\circ})}(\vn, &  D(N_{0},\n_{n}^{\circ})\otimes_{D^{\La}(N_{0},K)}D^{\sm}(N_{0},K)) \\
    & \cong \Tor_{\star}^{D(N_{0},\n_{n}^{\circ})}(\vn, D^{\sm}(N_{0},\n_{n}^{\circ})).
\end{split}
\end{equation} 

Suppose $P_{n,\star}$ is a finite free resolution of $\vn$ by $\dn$-modules, which can be viewed as a free resolution of abstract $D(N_{0},\n_{n}^{\circ})$-modules. Equation (\ref{eqn:tor}) shows the homology of the complex $(P_{n,\star})_{\N}$ can be identified with $H_{\star}(\N,\vn)$. Recall that taking $N_{0}$-coinvariants is exact on smooth $N_{0}$ representations \cite[Prop 3.2.3]{Casselman}. Taking into account of the fact that $(P_{n,\star})_{N_{0}}=((P_{n,\star})_{\N})_{N_{0}}$, we see that \begin{equation}
    H_{\star}(N_{0},\vn) \cong H_{\star}(\N,\vn)_{N_{0}}. 
\end{equation} By Poincar\'e duality \cite[Thm 6.10]{Knapp}, the homology groups $H_{\star}(\N,\vn)$ and $H_{\star}(\N,V^\prime_b)$ are dual to the cohomology groups $H^{k-\star}(\N,V_{\h_n^\circ\An})$ and $H^{k-\star}(\N,V)$ respectively, where $k$ is the dimension of $\N$. Since $H^{k-\star}(\N,V)$ is the direct limit of $H^{k-\star}(\N,V_{\h_n^\circ\An})$, it follows that $H_{\star}(\N,V^\prime_b)$ is the inverse limit of $H_{\star}(\N,\vn)$. 
\end{proof}

\begin{remark}\label{rmk:lie_cohom}
From theorem \ref{thm:lie_cohom}, we may want to conclude a similar comparison theorem for cohomology groups of the form: 
\begin{equation}\label{eq:lie_cohom}
    \underset{\underset{n}{\ra}}{\lim}\,H^{\star}(N_{0},V_{\h_{n}^{\circ}\An}) \cong H^{\star}(\N,V)^{N_{0}}.
\end{equation}
However, while the cohomology groups $H^{\star}(N_{0},V_{\h_{n}^{\circ}\An})$ are Hausdorff by definition (it is the dual of the Hausdorff completion of the homology groups), the Lie algebra cohomology need not be Hausdorff. In the case where the Lie algebra cohomology is Hausdorff, then (\ref{eq:lie_cohom}) does hold by applying Poincar\'e duality to the result of theorem \ref{thm:lie_cohom}. 
\end{remark}


\subsection{Hecke action}\label{sec:Hecke}
Let $Z^{+}$ be the submonoid of $Z_{M}^{+}$ consisting of elements $z$ satisfying $z^{-1}\overline{N}_{0}z\sub \overline{N}_{0}$ and $zN_{0}z^{-1}\sub N_{0}$. A variant of \cite[Prop. 3.3.2 (i)]{Jacquet1} show that $Z^{+}$ generates $Z_{M}$ as a group. Let $z$ be a fixed element of $Z^{+}$. The goal of this subsection is to define a $D(M_{0},\m_{n}^{\circ})$-linear action on the homology groups $H_{\star}(N_{0},\vn)$, which extends the Hecke operator $\pi_{N_{0},z}$ on $V^{N_{0}}$ given by (\ref{eqn:original_Hecke}).

Let $z\in Z^{+}$. For each $n\geq 0$, let $\h(z)_{n}$ and $\h(z^{-1})_{n}$ be the rigid analytic subgroups of $\h$ corresponding to the $\zp$-sublattices $\Ad_{z^{-1}}(\H_{n})\cap \H_{n}$ and $\Ad_{z}(\H_{n})\cap \H_{n}$ respectively. Let $H(z)_{n}=\h(z)_{n}(\qp)$ and $H(z^{-1})_{n}=\h(z^{-1})_{n}(\qp)$. By considering the Iwahori decompositions of rigid analytic groups\begin{align*}
    z^{-1}\h_{n}z &= z^{-1}\overline{\n}_{n}z \times \m_{n} \times z^{-1}\n_{n}z \textnormal{ and}\\
    z\h_{n}z^{-1} &= z\overline{\n}_{n}z^{-1} \times \m_{n} \times z\n_{n}z^{-1},
\end{align*} and conditions (5) and (6) of Proposition \ref{prop:good_grps}, the rigid analytic groups $\h(z)_{n}$ and $\h(z^{-1})_{n}$ admit Iwahori decompositions \begin{equation}\label{eqn:Iwahori_H(z)}
\begin{split}
    \h(z)_{n} &=  z^{-1}\overline{\n}_{n}z\times \m_{n} \times \n_{n}  \textnormal{ and}\\
    \h(z^{-1})_{n} &= \overline{\n}_{n} \times \m_{n} \times z\n_{n}z^{-1}.
\end{split}
\end{equation} Similarly, there are Iwahori decompositions \begin{align*}
    H(z)_{n} &= z^{-1}\overline{N}_{n}z \times M_{n} \times N_{n} \textnormal{ and}\\
    H(z^{-1})_{n} &= \overline{N}_{n}\times M_{n} \times zN_{n}z^{-1} .
\end{align*}

Let $n\geq 0$ be large enough so that $H_{n}\sub H(z)_{0}$ and $H_{n}\sub H(z^{-1})_{0}$. Define the restriction map to be the composition \begin{equation}\label{eqn:res}
    \begin{split}
        \res : H_{\star}(N_{0},\vn) & \ra H_{\star}(N_{0},\ind_{K[zN_{0}z^{-1}]}^{K[N_{0}]}\vn) \\
        & \cong H_{\star}(N_{0},K[N_{0}]\otimes_{K[zN_{0}z^{-1}]}\vn) \\
        & \cong H_{\star}(zN_{0}z^{-1},\vn).
    \end{split}
\end{equation}

\begin{proposition}\label{prop:Hecke_isom}
There is an isomorphism of topological $D(M_{0},\m_{n}^{\circ})$-modules \[
    H_{\star}(zN_{0}z^{-1},\vn) \cong H_{\star}(zN_{0}z^{-1},\vln).
\]
\end{proposition}
\begin{proof}
This follows from the second assertion of Proposition \ref{prop:H(V_L)->H(V_H)} with the observation that $\h_{n}$ and $\h(z^{-1})_{n}$ have the same ``opposite parabolic part". 
\end{proof}

The Iwahori decomposition of the rigid analytic groups $H(z)_{n}$ and $H(z^{-1})_{n}$ show that there is a $D(M_{0},\m_{n}^{\circ})$-linear isomorphism \[
    z^{-1}D(H(z^{-1})_{0},\h(z^{-1})_{n}^{\circ})z \cong D(H(z)_{0},\h(z)_{n}^{\circ}).
\] A simple argument on the resolution level then shows that conjugation by $z^{-1}$ induces a topological $D(M_{0},\m_{n}^{\circ})$-module isomorphism \begin{equation}\label{eqn:vln_vhn}
    H_{\star}(zN_{0}z^{-1},\vln) \cong H_{\star}(N_{0},\vhn). 
\end{equation} 

For each $n$ large enough so that $H_{n}\sub H(z)_{0}$ and $H_{n}\sub H(z^{-1})_{0}$, let $\pi_{N_{0},n,z}^{\prime}$ be the $D(M_{0},\m_{n}^{\circ})$-linear endomorphism of $H_{\star}(N_{0},\vn)$ defined by composing the following \begin{equation}\label{eqn:Hecke}\begin{split}
    H_{\star}(N_{0},\vn) & \overset{\res,\, (\ref{eqn:res})}{\longrightarrow} H_{\star}(zN_{0}z^{-1},\vn) \overset{\textnormal{Prop. }\ref{prop:Hecke_isom}}{\cong} H_{\star}(zN_{0}z^{-1},\vln) \\
    & \overset{(\ref{eqn:vln_vhn})}{\cong}H_{\star}(N_{0},\vhn) \overset{Prop. \ref{prop:H(V_L)->H(V_H)}} {\longrightarrow} H_{\star}(N_{0},\vn)
\end{split}\end{equation} with the multiplication by $\delta_{P}(z)$ map on $H_{\star}(N_{0},\vn)$. It is not hard to check that the natural transition maps $H_{\star}(N_{0},(V_{\h^{\circ}_{n+1}\An})^{\prime}_{b})\ra H_{\star}(N_{0},\vn)$ are $Z^{+}$-equivariant. 

\begin{remark}\label{rmk: Hecke_on_terms}
By our construction, we can actually define $\pi_{N_{0},n,z}^{\prime}$ on the terms of the complex computing the homology group. More specifically, suppose $\dn^{r_{n,\star}}\ra \vn$ is a free resolution of $\vn$. By the Iwahori decomposition of $H(z)_{0}$ and $H(z^{-1})_{0}$ (see (\ref{eqn:Iwahori_H(z)})), all of the following indices are equal \[
    \left[H_{0}:H(z^{-1})_{0}\right] = \left[N_{0}:zN_{0}z^{-1}\right] = \left[H_{0}:H(z)_{0}\right].
\] Let $u$ be this constant integer. Let $\tilde{\pi}_{N_{0},n,z}^{\prime}$ be the following composition of maps \begin{equation}\begin{split}
    \tilde{\pi}_{N_{0},n,z}^{\prime}: \dn^{r_{n,\star}}_{N_{0}} \overset{\res}{\longrightarrow} & \dn^{r_{n,\star}}_{zN_{0}z^{-1}} \overset{}{\cong} D(H(z^{-1})_{0},\h_{n}^{\circ})^{r_{n,\star} u}_{zN_{0}z^{-1}} \\
    \overset{Prop. \ref{prop:Hecke_isom}}{\cong} & D(H(z^{-1})_{0},\h(z^{-1})_{n}^{\circ})^{r_{n,\star} u}_{zN_{0}z^{-1}} \overset{z^{-1}(\cdot)z}{\cong} D(H(z)_{0},\h(z)_{n}^{\circ})^{r_{n,\star} u}_{N_{0}} \\
    \overset{}{\longrightarrow} & D(H(z)_{0},\h_{n}^{\circ})^{r_{n,\star} u}_{N_{0}} \overset{}{\cong} D(H_{0},\h_{n}^{\circ})^{r_{n,\star}}_{N_{0}}.
\end{split}\end{equation} Then $\tilde{\pi}_{N_{0},n,z}^{\prime}$ induces $\frac{1}{\delta_{P}(z)}\pi_{N_{0},n,z}^{\prime}$ on the homology groups $H_{\star}(N_{0},\vn)$. 
\end{remark}

The rest of the subsection is devoted to checking that $\pi_{N_{0},0,z}^{\prime}$ is compatible with $\pi_{N_{0},z}^{\prime}$ (the dual of (\ref{eqn:original_Hecke})) in degree 0. More specifically, we would like to check that the following diagram commutes \begin{equation}\label{diag: Hecke_compatibility1}
\begin{tikzcd}
    (V^{N_{0}})^{\prime}_{b} \arrow[r,"\cong"] \arrow[d,"\pi_{N_{0},z}^{\prime}"] & (V^{\prime}_{b})_{N_{0}} \arrow[r,"v^{\prime}\map 1 \otimes v^{\prime}"] & (\dn \otimes_{D^{\La}(H_{0},K)} V^{\prime}_{b})_{N_{0}} \arrow[r,"\cong"] & (\vn)_{N_{0}} \arrow[d,"\pi_{N_{0},0,z}^{\prime}"] \\
    (V^{N_{0}})^{\prime}_{b} \arrow[r,"\cong"] & (V^{\prime}_{b})_{N_{0}} \arrow[r,"v^{\prime}\map 1 \otimes v^{\prime}"] & (\dn \otimes_{D^{\La}(H_{0},K)} V^{\prime}_{b})_{N_{0}} \arrow[r,"\cong"] & (\vn)_{N_{0}}
\end{tikzcd}. 
\end{equation} We use some ideas similar to \cite[Lem. 4.2.19, Lem. 4.2.21]{analytic}.  

Let \begin{equation*}
    \lambda: \,H_{0} \ra \overline{N}_{0} \times M_{0} = \overline{P}_{0} \textnormal{ and }\rho:  \,H_{0} \ra N_{0}
\end{equation*} and be the natural projections coming from the Iwahori decomposition of $H_{0}$. Recall that the map $K[H_{0}]\ra \dn$ by sending an element $h\in H_{0}$ to the dirac distribution $\delta_{h}$ is dense. The Hecke operator $\pi^{\prime}_{N_{0},0,z}$ acts on an element $\delta_{h}\otimes v^{\prime}\in (\dn \otimes_{D^{\La}(H_{0},K)} V^{\prime}_{b})_{N_{0}}$ by the following composition: \begin{equation}\begin{split}
    \delta_{h}\otimes v^{\prime} \overset{\res}{\map} & \, \delta_{P}(z)\underset{x\in N_{0}/zN_{0}z^{-1}}{\Sigma} \delta_{x^{-1}h}\otimes x^{-1}v^{\prime}  \overset{Prop. \ref{prop:Hecke_isom}}{\longrightarrow} \delta_{P}(z)\underset{x\in N_{0}/zN_{0}z^{-1}}{\Sigma} \delta_{\lambda(x^{-1}h)}\otimes x^{-1}v^{\prime} \\
    \overset{z}{\ra} & \,\delta_{P}(z)\underset{x\in N_{0}/zN_{0}z^{-1}}{\Sigma} \delta_{z^{-1}\lambda(x^{-1}h)z}\otimes z^{-1}x^{-1}v^{\prime}.
\end{split}\end{equation} 
It now suffices to check that the following diagram commutes: \begin{equation}\label{diag: Hecke_compatibility2}
\begin{tikzcd}[column sep=large]
    K[H_{0}] \otimes_{K} V^{\prime}_{b} \arrow[r,"h\otimes v^{\prime}\map hv^{\prime}"] \arrow[d,"\pi_{N_{0},0,z}^{\prime}"] & (V^{\prime}_{b})_{N_{0}} \arrow[r,"\cong"] & (V^{N_{0}})^{\prime}_{b} \arrow[d,"\pi_{N_{0},z}^{\prime}"]  \\
    K[H_{0}] \otimes_{K} V^{\prime}_{b} \arrow[r,"h\otimes v^{\prime}\map hv^{\prime}"] & (V^{\prime}_{b})_{N_{0}} \arrow[r,"\cong"] & (V^{N_{0}})^{\prime}_{b}
\end{tikzcd}.
\end{equation} 

Let $\left<\cdot,\cdot\right>$ denote the duality pairing between $V$ and $V^{\prime}_{b}$. The goal is to prove that there is an equality \[
    \left<\delta_{P}(z)\underset{x\in N_{0}/zN_{0}z^{-1}}{\Sigma} \delta_{z^{-1}\lambda(x^{-1}h)z}\otimes z^{-1}x^{-1}v^{\prime},v\right>=\left<h^{-1}v^{\prime},\pi_{N_{0},z}(v)\right>
\] for all $h\in H_{0}$, $v\in V^{N_{0}}$ and $v^{\prime}\in V^{\prime}_{b}$. A simple computation shows that \begin{equation}
\begin{split}
    \left\langle \delta_{z^{-1}\lambda\left(x^{-1}h\right)z}\otimes z^{-1}x^{-1}v^{\prime},v\right\rangle &= \left\langle z^{-1}\lambda(x^{-1}h)^{-1}zz^{-1}x^{-1}v^{\prime},v\right\rangle \\
	&=	\left\langle z^{-1}\lambda(x^{-1}h)^{-1}x^{-1}v^{\prime},v\right\rangle \\ 
	&=	\left\langle v^{\prime},x\lambda(x^{-1}h)zv\right\rangle \\
	&=	\left\langle h^{-1}v^{\prime},h^{-1}x\lambda(x^{-1}h)zv\right\rangle .
\end{split}
\end{equation} The following lemma then finishes the proof. 

\begin{lemma}
As $x$ ranges over the right coset representatives of $zN_{0}z^{-1}$ in $N_{0}$, $h^{-1}x\lambda(x^{-1}h)$ ranges through a set of right coset representatives of $zN_{0}z^{-1}$ in $N_{0}$ as well.
\end{lemma}
\begin{proof}
By definition, $h=\lambda(h)\rho(h)$ for all $h\in H_{0}$. In particular, \[        
    h^{-1}x\lambda\left(x^{-1}h\right)=\rho\left(x^{-1}h\right)^{-1}
\] is indeed an element of $N_{0}$. Additionally, given $x,y\in N_{0}$, there is an identity \begin{equation}\label{eqn:cosets_N_zNz^{-1}}
    \rho\left(y^{-1}h\right)\rho\left(x^{-1}h\right)^{-1}=\lambda\left(y^{-1}h\right)^{-1}y^{-1}x\lambda\left(x^{-1}h\right).
\end{equation}
\begin{claim}
For any $p,p^{\prime}\in\overline{P}_{0}$, there is an inclusion $p\left(zN_{0}z^{-1}\right)p^{\prime}\cap N_{0}\sub zN_{0}z^{-1}$.
\end{claim}
\begin{proof}
This follows from the fact that $p\left(zN_{0}z^{-1}\right)p^{\prime}$ is element of $H_{0}$ and $H(z^{-1})_{0}$. From the Iwahori decomposition of these two groups, $H_{0}\cap H(z^{-1})_{0}\cap N_{0}=zN_{0}z^{-1}$.
\end{proof}

Looking back at equation (\ref{eqn:cosets_N_zNz^{-1}}), we see that if $y^{-1}x\in zN_{0}z^{-1}$ then the claim shows that $\rho\left(y^{-1}h\right)\rho\left(x^{-1}h\right)^{-1}\in zN_{0}z^{-1}$. Conversely, if $\rho\left(y^{-1}h\right)\rho\left(x^{-1}h\right)^{-1}\in zN_{0}z^{-1}$ then \[
    y^{-1}x\in\lambda\left(y^{-1}h\right)\left(zN_{0}z^{-1}\right)\lambda\left(x^{-1}h\right)^{-1}\sub zN_{0}z^{-1}
\] as well.
\end{proof}

\subsection{Definition of $H^{\star}J_{P}(V)$}\label{sec:def}

\begin{lemma}\label{lem:compact_quotient}
Suppose $f:\,V\ra W$ is a compact continuous linear map between Hausdorff locally convex $K$-vector spaces (in the sense of \cite[\textsection 16]{Schneider_nfs}). Suppose there is a commutative diagram of Hausdorff locally convex $K$-vector spaces \begin{equation*}
    \begin{tikzcd}
	V \arrow[r,"f"] \arrow[d,"p"] & W \arrow[d,"q"] \\
	Y \arrow[r,"\bar{f}"] & Z
    \end{tikzcd},
\end{equation*} where $p$ is a quotient map. Then $\overline{f}$ is compact.
\end{lemma}
\begin{proof}
Let $L$ be an open lattice of $V$ such that $\overline{f(L)}$ is bounded and c-compact. By Lemma 12.1, Lemma 12.4 and Proposition 12.7 of \cite{Schneider_nfs}, the image $q\left(\overline{f(L)}\right)$ is closed, bounded and c-compact. Therefore, $p(L)$ is an open lattice in $Y$ such that $\overline{\overline{f}\left(p(L)\right)} = \overline{q\left(\overline{f(L)}\right)}=q\left(\overline{f(L)}\right)$ is bounded and c-compact, and so $\overline{f}$ is compact.
\end{proof}

By the choice of good analytic open subgroups $\{H_{n}\}_{n\geq0}$ of $G$, each of the rigid analytic inclusion $\h_{n+1}\sub \h_{n}$ is relatively compact. Therefore, the natural map $D(H_{0},\h_{n+1}^{\circ})\ra \dn$ is compact. Lemma \ref{lem:compact_quotient} and \cite[Rem. 16.7]{Schneider_nfs} show that the transition maps $\widehat{H}_{\star}(N_{0},(V_{\h_{n+}^{\circ}\An})^{\prime}_{b}) \ra \widehat{H}_{\star}(N_{0},(V_{\h_{n}^{\circ}\An})^{\prime}_{b})$ defined via proposition \ref{prop:H(V_L)->H(V_H)} are also compact. Hence, $\underset{\underset{n}{\la}}{\lim}\widehat{H}_{\star}(N_{0},(V_{\h_{n}^{\circ}\An})^{\prime}_{b})$ is a nuclear Fr\'echet space. It also comes equipped with an action of $Z^{+}$ (defined in \textsection \ref{sec:Hecke}) induced from the operators $\pi_{N_{0},n,z}^{\prime}$ for each $z\in Z_{M}^{+}$ and $n$ large enough. Lemma 3.2.29 of \cite{Jacquet1} and the discussion following the lemma show that the strong dual of $\underset{\underset{n}{\la}}{\lim}\widehat{H}_{\star}(N_{0},(V_{\h_{n}^{\circ}\An})^{\prime}_{b})$ is an object of $\rep_{\textnormal{la,c}}^{z}(Z^{+})$. 

Motivated by equation (\ref{eqn:fs_dual}) and \cite[Lem. 3.2.19]{Jacquet1}, for an admissible locally analytic $G$-representation $V\in\rep_{\textnormal{ad}}(G)$ we make the following definition. 
\begin{definition}
We define the derived Jacquet-Emerton module $H^{\star}J_{P}(V)$ to be the strong dual of \begin{equation}\label{eq:def_of_DJP}
    \left(H^{\star}J_{P}(V)\right)^{\prime}_{b} := C^{\an}(\widehat{Z}_{M},K)\underset{K[Z^{+}]}{\widehat{\otimes}} \underset{\underset{n}{\la}}{\lim}\widehat{H}_{\star}(N_{0},(V_{\h_{n}^{\circ}\An})^{\prime}_{b}). 
\end{equation}
\end{definition} 

For this definition to be reasonable, we need to show that the right hand side of (\ref{eq:def_of_DJP}) is reflexive. This is true, as we show in Section \ref{sec:properties} that $H^{\star}J_P$ is a family of functors from $\rep_{\textnormal{ad}}(G)$ to $\rep_{\textnormal{es}}(M)$. By \cite[Prop. 1.1.32]{analytic} there is an isomorphism \begin{equation*}
    H^{\star}J_{P}(V) \cong \left(\underset{\underset{n}{\ra}}{\lim}H^{\star}(N_{0},V_{\h_{n}^{\circ}\An}) \right)_{\textnormal{fs}}.
\end{equation*} In degree zero, we observe that \begin{equation*}
    \left(H^{0}J_{P}(V)\right)^{\prime}_{b} \cong C^{\an}(\widehat{Z}_{M},K)\underset{K[Z^{+}]}{\widehat{\otimes}} \underset{\underset{n}{\la}}{\lim}((V_{\h_{n}^{\circ}\An})^{\prime}_{b})_{N_{0}} \cong C^{\an}(\widehat{Z}_{M},K)\underset{K[Z^{+}]}{\widehat{\otimes}} (V^{\prime}_{b})_{N_{0}}.
\end{equation*} Therefore, there is an isomorphism $H^{0}J_{P}(V)\cong (V^{N_{0}})_{\textnormal{fs}} \cong J_{P}(V)$. 

Let $u\in Z^{+}$ be an element satisfying the properties that $u^{-1}\overline{N}_{n}u\sub \overline{N}_{n+1}$ for all $n\geq 0$, and that $u^{-1}$ and $Z^{+}$ generates $Z_{M}$. Define $Y^{+}$ (resp. $Y$) to be the submonoid (resp. subgroup) of $Z_{M}$ generated by $u$. Fix an exhaustive increasing sequence $\{\widehat{Y}_{n}\}_{n\geq0}$ of admissible open affinoid subdomains of $\widehat{Y}_{M}$ with the property that the inclusions $Y_{n}\sub Y_{n+1}$ are relatively compact. Let $C^{\an}(\widehat{Y}_{n},K)^{\dagger}$ denote the space of overconvergent rigid analytic functions on $\widehat{Y}_{n}$. Then there is a topological isomorphism \begin{equation}\label{eqn:C(Z)=limC(Zn)}
    C^{\an}(\widehat{Y},K) \cong \underset{\underset{n}{\la}}{\lim}C^{\an}(\widehat{Y}_{n},K)^{\dagger}.
\end{equation} Proposition 3.2.28 of \cite{Jacquet1}, Proposition 1.1.29 of \cite{analytic} and Lemma 19.10 of \cite{Schneider_nfs} show that there is an isomorphism \begin{equation}\label{eqn:HJp(V)_H}
    \left(H^{\star}J_{P}(V)\right)^{\prime}_{b} := \underset{\underset{n}{\la}}{\lim} C^{\an}(\widehat{Y}_{n},K)^{\dagger}\underset{K[u]}{\widehat{\otimes}} H_{\star}(N_{0},(V_{\h_{n}^{\circ}\An})^{\prime}_{b}),
\end{equation} where the limit is over all $n\geq 0$ large enough so that $\h_{n}\sub \h(z)_{0}$ and $\h_{n}\sub \h(z^{-1})_{0}$. 

\begin{proposition}
Suppose $V\in \rep_{\textnormal{ad}}(G)$. Then the assignment $V\map H^{\star}J_{P}(V)$ is a family of functors from $\rep_{\textnormal{ad}}(G)$ to $\rep_{\textnormal{la,c}}^{z}(M)$.
\end{proposition}
\begin{proof}
Suppose $f:\,V\ra W$ is a continuous map between admissible locally analytic $G$-representations. Then for each $n\geq 0$, there is a continuous map between finitely presented topological $\dn$-modules \begin{equation*}
    f_{n}^{\prime}:\, (W_{\h_{n}^{\circ}\An})_{b}^{\prime} \ra \vn. 
\end{equation*} Passing to homology, there are continuous $D(M_{0},\m_{n}^{\circ})$-linear maps \begin{equation*}
    (f_{n}^{\prime})_{\star}:\, H_{\star}(N_{0},(W_{\h_{n}^{\circ}\An}))_{b}^{\prime} \ra H_{\star}(N_{0},\vn) 
\end{equation*} which can easily be checked to be $u$-equivariant. Since the action of $u$ commutes with $D(M_{0},\m_{n}^{\circ})$, after taking limits and $Y$-finite slope parts, there are continuous  $D^{\La}(M_{0},K)$-linear maps $(H^{\star}J_{P}(W))^{\prime}_{b}\ra (H^{\star}J_{P}(V))^{\prime}_{b}$. This proves that $H^{\star}J_{P}(V)$ is a family of functors.

Suppose $V\in \rep_{\textnormal{ad}}(G)$, then $H^{\star}J_{P}(V)\in \rep_{\textnormal{la,c}}^{z}(Z_{M})$ \cite[Prop. 3.2.4]{Jacquet1}. Fix $m\in M^{+}$. For each $n\geq 0$, let $\h(m^{-1})_{n}$ (resp. $\h(m)_{n}$) be the good analytic open subgroup of $G$ associated to the $\zp$-lattice $\overline{\N}_{n}\oplus \M_{n} \oplus \Ad_{m}(\N_{n})$ (resp. $\Ad_{m^{-1}}(\overline{\N}_{n}\oplus \M_{n}) \oplus \N_{n}$). For each $n$ large enough so that $H_{n}\sub H(m^{-1})_{0}$ and $H_{n}\sub H(m)_{0}$. Define $\psi_{n,m}$ to be the composition \begin{equation*}\begin{split}
    \psi_{n,m}: &\, \widehat{H}_{\star}(N_{0},(V_{\h_{n}^{\circ}\An})^{\prime}_{b}) \overset{\textnormal{res, }(\ref{eqn:res})}{\longrightarrow} \widehat{H}_{\star}(mN_{0}m^{-1},(V_{\h_{n}^{\circ}\An})^{\prime}_{b})  \\ 
    \overset{\textnormal{Prop. } \ref{prop:H(V_L)->H(V_H)}}{\cong} &  \widehat{H}_{\star}(mN_{0}m^{-1},(V_{\h(m^{-1})_{n}^{\circ}\An})^{\prime}_{b})\overset{m^{-1}}{\cong} \widehat{H}_{\star}(N_{0},(V_{\h(m)_{n}^{\circ}\An})^{\prime}_{b}).
\end{split} 
\end{equation*} By taking projective limit and composing with the isomorphism coming from Corollary \ref{cor:same_inv_limit} as well as the multiplication by $\delta_{P}(m)$-endomorphism, there is a map \begin{equation*}
    \psi_{z}:\, \underset{\underset{n}{\la}}{\lim} \widehat{H}_{\star}(N_{0},(V_{\h_{n}^{\circ}\An})^{\prime}_{b}) \ra \underset{\underset{n}{\la}}{\lim} \widehat{H}_{\star}(N_{0},(V_{\h_{n}^{\circ}\An})^{\prime}_{b}).
\end{equation*} We define this to be the action of $M^{+}$ on $\underset{\underset{n}{\la}}{\lim} \widehat{H}_{\star}(N_{0},(V_{\h_{n}^{\circ}\An})^{\prime}_{b})$. 

Notice that if $z\in Z^{+}$, then the Hecke action $\pi_{N_{0},n,z}^{\prime}$ is exactly the composition of $\psi_{n,z}$ with the map $\widehat{H}_{\star}(N_{0},(V_{\h(z)_{n}^{\circ}\An})^{\prime}_{b})\ra \widehat{H}_{\star}(N_{0},(V_{\h_{n}^{\circ}\An})^{\prime}_{b})$ and the multiplication by $\delta_{P}(z)$-endomorphism. Thus, the action of $z$ on $\underset{\underset{n}{\la}}{\lim} \widehat{H}_{\star}(N_{0},(V_{\h_{n}^{\circ}\An})^{\prime}_{b})$ coming from $\pi_{N_{0},z}^{\prime}$ and $\psi_{z}$ agree. Proposition 3.3.2 and Proposition 3.3.6 of \cite{Jacquet1} imply that this defines an action of $M$ on $H^{\star}J_{P}(V)$. 

It remains to show that the action of $M$ on $H^{\star}J_{P}(V)$ is locally analytic. Since $M_{0}$ is a compact open subgroup of $M$, it is equivalent to showing that the action of $M_{0}$ is locally analytic. Since $M_{0}$ normalizes $N_{0}$ and $\overline{N}_{0}$, it is easy to see that for all $m\in M_{0}$, the action of $m$ given by $\psi_{m}$ is simply the multiplication by $m$ map on $\underset{\underset{n}{\ra}}{\lim} \widehat{H}^{\star}(N_{0},V_{\h_{n}^{\circ}\An})$ and the multiplication by the dirac distribution $\delta_{m}$ on $\underset{\underset{n}{\la}}{\lim} \widehat{H}_{\star}(N_{0},(V_{\h_{n}^{\circ}\An})^{\prime}_{b})$. Since $\underset{\underset{n}{\la}}{\lim} \widehat{H}_{\star}(N_{0},(V_{\h_{n}^{\circ}\An})^{\prime}_{b})$ is a topological $D^{\La}(M_{0},K)$-module, the action of $M_{0}$ on $\underset{\underset{n}{\ra}}{\lim} \widehat{H}^{\star}(N_{0},V_{\h_{n}^{\circ}\An})$ is locally analytic \cite[Prop. 2.2]{summary}. Since the Hecke action of $Z^{+}$ is $M$-linear, it follows from \cite[Lem. 3.2.7]{Jacquet1} that the action of $M$ on $H^{\star}J_{P}(V)$ is also locally analytic. 
\end{proof}

The rest of the section is dedicated to showing that the definition of the functor $H^{\star}J_{P}$ is invariant of the various choices we have made in its construction. A number of the arguments are similar to those used in \cite{Jacquet1}.

Suppose $M^{\prime}$ is another choice of lift of the Levi quotient of $P$. Then there is some $x\in N$ such that $M^{\prime}=xMx^{-1}$. For all $n\geq0$, let $\h_{n}^{\prime}=x\h_{n}x^{-1}$ $\m_{n}^{\prime}=x\m_{n}x^{-1}$, $\n_{n}^{\prime}=x\m_{n}x^{-1}$, $H_{n}^{\prime}=\h_{n}^{\prime}(\qp)$, $M_{n}^{\prime}=\m_{n}^{\prime}(\qp)$ and $N_{n}^{\prime}=\n_{n}^{\prime}(\qp)$. Let $(M^{+})^{\prime}=xM^{+}x^{-1}$, then \begin{equation*}
    (M^{+})^{\prime} = \{ m^{\prime}\in M^{\prime}:\, m^{\prime}N_{0}^{\prime}(m^{\prime})^{-1}\sub N_{0}^{\prime}\}. 
\end{equation*} Given $V\in \rep_{\textnormal{ad}}(G)$, we can define $\underset{\underset{n}{\la}}{\lim} \widehat{H}_{\star}(N^{\prime}_{0},(V_{\h_{n}^{\prime,\circ}\An})^{\prime}_{b})$ and an action of $(M^{+})^{\prime}$ on the space in an analogous way. 

\begin{proposition}
Suppose $V\in\rep_{\textnormal{ad}}(G)$. Multiplication by $x$ induces an isomorphism \begin{equation}
    \left(\underset{\underset{n}{\ra}}{\lim}H^{\star}(N_{0},(V_{\h_{n}^{\circ}\An})^{\prime}_{b}) \right)_{\textnormal{fs}} \cong \left(\underset{\underset{n}{\ra}}{\lim}H^{\star}(N_{0}^{\prime},(V_{\h_{n}^{\prime,\circ}\An})^{\prime}_{b}) \right)_{\textnormal{fs}}
\end{equation} in the category $\rep_{\textnormal{la,c}}^{z}(M)$ (where $M^{\prime}$ is identified with $M$ via conjugation by $x$). 
\end{proposition}
\begin{proof}
Multiplication by $x$ induces isomorphisms of compact type $K$-algebras \begin{equation*}
    D(H_{0},\h_{n}^{\circ})\cong D(H_{0}^{\prime},\h_{n}^{\prime,\circ}) \textnormal{ and } D(H_{0},\h_{n}^{\circ})_{N_{0}}\cong D(H_{0}^{\prime},\h_{n}^{\prime,\circ})_{N_{0}^{\prime}}. 
\end{equation*} Therefore, all the homology groups are naturally identified and the proposition follows immediately. 
\end{proof}

\begin{proposition}
Suppose $L_{0}$ is a good analytic open subgroup of $H_{0}$. For each $n\geq 0$, let $L_{n}=H_{n}\cap L_{0}$, $M_{n}^{\prime} = M\cap L_{n}$, $\overline{P}_{n}^{\prime} = \overline{P}\cap L_{n}$, $N_{n}^{\prime}= N\cap L_{n}$ and let $\L_{n}$, $\n_{n}^{\prime}$, $\overline{\P}_{n}^{\prime}$ and $\n^{\prime}_{n}$ be their respective rigid analytic Zariski closures in $\L_{n}$. Then there is a natural $M$-equivariant topological isomorphism \begin{equation}\label{eqn:inv_of_choice_fs}
    \left(\underset{\underset{n}{\ra}}{\lim}H^{\star}(N_{0},V_{\h_{n}^{\circ}\An})^{\prime}_{b}) \right)_{\textnormal{fs}} \cong \left(\underset{\underset{n}{\ra}}{\lim}H^{\star}(N_{0}^{\prime},V_{\L_{n}^{\circ}\An})^{\prime}_{b}) \right)_{\textnormal{fs}}
\end{equation}
\end{proposition}
\begin{proof}
Let $(M^{+})^{\prime}$ be the submonoid of $M$ that conjugates $N_{0}^{\prime}$ into itself. Let $Z_{M^{\prime}}^{+}=(M^{+})^{\prime}\cap Z_{M}$. Let $(Z^{+})^{\prime}$ be the submonoid of $Z_{M^{\prime}}^{+}$ consisting of elements $z$ such that $z^{-1}\overline{N}_{0}^{\prime}z\sub \overline{N}_{0}^{\prime}$. Let $u^{\prime}\in (Z^{+})^{\prime}$ be an element satisfying the properties that $(u^{\prime})^{-1}\overline{N}^{\prime}_{n}u^{\prime}\sub \overline{N}^{\prime}_{n+1}$ for all $n\geq 0$, and that $(u^{\prime})^{-1}$ and $(Z^{+})^{\prime}$ generates $Z_{M}$. Define $(Y^{+})^{\prime}$ (resp. $Y^{\prime}$) to be the submonoid (resp. subgroup) of $Z_{M}$ generated by $u^{\prime}$. 

Choose $\alpha\in Y^{+}$ such that $\alpha N_{0}\alpha^{-1}\sub N_{0}^{\prime}$. For each $n$ large enough so that $H_{n}\sub L_{0}$, define $\varphi_{n}$ to be the following composition of continuous $D(M_{0}^{\prime},\m_{n}^{\prime,\circ})$-linear maps \begin{equation*}\begin{split}
    \varphi_{n}: &\, H_{\star}(N_{0},(V_{\h_{n}^{\circ}\An})^{\prime}_{b}) \overset{\textnormal{res, }(\ref{eqn:res})}{\longrightarrow} H_{\star}(N_{0}^{\prime},(V_{\h_{n}^{\circ}\An})^{\prime}_{b}) \\
    & \overset{\textnormal{Prop. }\ref{prop:H(V_L)->H(V_H)}}{\cong} D(\overline{P}^{\prime}_{0},\overline{\P}_{n}^{\circ})\underset{D(\overline{P}^{\prime}_{0},\overline{\P}^{\prime,\circ}_{n})}{\widehat{\otimes}} H_{\star}(N_{0}^{\prime},(V_{\L_{n}^{\circ}\An})^{\prime}_{b}).
\end{split}
\end{equation*}  For each $n$ large enough so that $H_{n}\sub H(\alpha)_{n}$ and $H_{n}\sub H(\alpha^{-1})_{n}$, let $\psi_{n}$ be the continuous $D(M_{0}^{\prime},\m_{n}^{\prime,\circ})$-linear map defined by the composition \begin{equation*}\begin{split}
    \psi_{n}: & \, H_{\star}(N_{0}^{\prime},(V_{\L_{n}^{\circ}\An})^{\prime}_{b})  \overset{\textnormal{Prop. } \ref{prop:H(V_L)->H(V_H)}}{\longrightarrow}   H_{\star}(\alpha N_{0}\alpha^{-1},\vn) \overset{\textnormal{Prop. }\ref{prop:Hecke_isom}}{\cong} H_{\star}(\alpha N_{0}\alpha^{-1},(V_{\h(\alpha^{-1})_{n}^{\circ}\An})^{\prime}_{b}) \\
    & \overset{(\ref{eqn:vln_vhn})}{\cong}H_{\star}(N_{0},(V_{\h(\alpha)_{n}^{\circ}\An})^{\prime}_{b}) \overset{\textnormal{Prop. } \ref{prop:H(V_L)->H(V_H)}} {\longrightarrow} H_{\star}(N_{0},\vn).
\end{split} 
\end{equation*} Upon passing to Hausdorff completion and projective limits, the $\varphi_{n}$'s and $\psi_{n}$'s define continuous $D(M_{0}^{\prime},\m_{n}^{\prime,\circ})$-linear maps \begin{equation}\label{eqn:inv_of_choice} \begin{split}
    \varphi: & \, \underset{\underset{n}{\la}}{\lim} \widehat{H}_{\star}(N_{0},(V_{\h_{n}^{\circ}\An})^{\prime}_{b}) \ra \underset{\underset{n}{\la}}{\lim} \widehat{H}_{\star}(N_{0}^{\prime},(V_{\L_{n}^{\circ}\An})^{\prime}_{b}) \\
    \psi: & \,  \underset{\underset{n}{\la}}{\lim} \widehat{H}_{\star}(N_{0}^{\prime},(V_{\L_{n}^{\circ}\An})^{\prime}_{b}) \ra \underset{\underset{n}{\la}}{\lim} \widehat{H}_{\star}(N_{0},(V_{\h_{n}^{\circ}\An})^{\prime}_{b}).
\end{split} 
\end{equation} 

Let $C^{+}= Z^{+} \cap (Z^{+})^{\prime}$ and $D^{+}=M^{+}\cap (M^{+})^{\prime}$, then $C^{+}$ generates $Z_{M}$ as a group and $D^{+}$ generates $M$ as a group \cite[Prop. 3.3.2]{Jacquet1}. It is not hard to check that $\varphi$ and $\psi$ are both $C^{+}$-equivariant. Their compositions $\psi\circ \varphi$ and $\varphi \circ \psi$ are exactly $\frac{1}{\delta_{P}(\alpha)}\pi^{\prime}_{N_{0},\alpha}$ and $\frac{1}{\delta_{P}(\alpha)}\pi^{\prime}_{N_{0}^{\prime},\alpha}$ respectively. Proposition 3.2.20 of \cite{Jacquet1} shows that (\ref{eqn:inv_of_choice_fs}) is an $Z_{M}$-equivariant topological isomorphism. Furthermore, $\varphi$ and $\psi$ intertwines the $M^{+}\cap (M^{+})^{\prime}$-action on the source and target of (\ref{eqn:inv_of_choice}). Thus, (\ref{eqn:inv_of_choice_fs}) is in fact $M$-equivariant. 
\end{proof}


\vspace{3mm}
\section{Properties of $H^{\star}J_P$}\label{sec:properties}
In this section, we prove that $H^{\star}J_{P}(V)$ is a $\delta$-functor from $\rep_{\textnormal{ad}}(G)\ra \rep_{\textnormal{es}}(M)$. 

\subsection{Essential Admissibility}\label{sec:Essential_Admissibility}

The goal of this subsection is to prove that $H^{\star}J_{P}(V)$ is an essentially admissible locally analytic $M$-representation. We use a number of ideas introduced in \cite{Jacquet1}. In particular, we want to show that $H^{\star}J_{P}(V)$ satisfies the hyptheses of \cite[Prop. 3.2.24]{Jacquet1}, which are: 
\begin{enumerate}
    \item There is an element $z\in Z^{+}$ such that for all $n\geq 0$ sufficiently large, the Hecke operator $\pi_{N_{0},n,z}^{\prime}:\, \widehat{H}_{\star}(N_{0},(V_{\h_{n}^{\circ}\An})_{b}^{\prime}) \ra \widehat{H}_{\star}(N_{0},(V_{\h_{n}^{\circ}\An})_{b}^{\prime})$ factors through \begin{equation*}
        D(M_{0},\m_{n}^{\circ})\underset{D(M_{0},\m_{n+1}^{\circ})}{\widehat{\otimes}} \widehat{H}_{\star}(N_{0},(V_{\h_{n+1}^{\circ}\An})_{b}^{\prime}) \ra  \widehat{H}_{\star}(N_{0},(V_{\h_{n}^{\circ}\An})_{b}^{\prime})
    \end{equation*} to give a commutative diagram \begin{equation*}
    \begin{tikzcd}	
        D(M_{0},\m_{n}^{\circ})\underset{D(M_{0},\m_{n+1}^{\circ})}{\widehat{\otimes}} \widehat{H}_{\star}(N_{0},(V_{\h_{n+1}^{\circ}\An})_{b}^{\prime} \arrow[r] \arrow[d,"id\widehat{\otimes} \pi_{N_{0},n,z}^{\prime}"] & \widehat{H}_{\star}(N_{0},(V_{\h_{n}^{\circ}\An})_{b}^{\prime}) \arrow[d,"\pi_{N_{0},n,z}^{\prime}"] \arrow[dl] \\ 
	    D(M_{0},\m_{n}^{\circ})\underset{D(M_{0},\m_{n+1}^{\circ})}{\widehat{\otimes}} \widehat{H}_{\star}(N_{0},(V_{\h_{n+1}^{\circ}\An})_{b}^{\prime} \arrow[r] & \widehat{H}_{\star}(N_{0},(V_{\h_{n}^{\circ}\An})_{b}^{\prime})
    \end{tikzcd}.
    \end{equation*}
    \item For all $n\geq 0$ sufficiently large, the induced $D(M_{0},\m_{n}^{\circ})[Z^{+}]$-linear map \begin{equation*}
        D(M_{0},\m_{n}^{\circ})\underset{D(M_{0},\m_{n+1}^{\circ})}{\widehat{\otimes}} \widehat{H}_{\star}(N_{0},(V_{\h_{n+1}^{\circ}\An})_{b}^{\prime})\ra \widehat{H}_{\star}(N_{0},(V_{\h_{n}^{\circ}\An})_{b}^{\prime})
    \end{equation*} is $D(M_{0},\m_{n}^{\circ})$-compact in the sense of \cite[Def. 2.3.3]{Jacquet1}. 
\end{enumerate}

Fix $z\in Z^{+}$ satisfying the condition that for all $n\geq 0$, the inclusions $z^{-1}\overline{N}_{n}z\sub\overline{N}_{n}$ given by condition (7) factor through $\overline{N}_{n+1}\sub\overline{N}_{n}$. For each $n\geq 0$, let $H(z)_{n,n+1}$ be the good analytic open subgroup of $G$ corresponding to the $\zp$-Lie lattice $\Ad_{z^{-1}}(\H_{n})\cap \H_{n+1}$. The choice of $z$ and the sequence of good analytic open subgroups $\{H_{n}\}$ imply that each $\h(z)_{n,n+1}$ admits a rigid analytic Iwahori decomposition \[
    \h(z)_{n,n+1} \cong z^{-1}\overline{\n}_{n}z\times \m_{n+1} \times \n_{n+1}.
\] 
Since $H(z)_{n,n+1}$ is a normal subgroup of $H(z)_{0}$, we can define the topological $K$-algebra $D(H(z)_{0},\h(z)_{n,n+1}^{\circ})$ to be the strong dual of $C^{\La}(H(z)_{0},K)_{\h(z)_{n,n+1}\An}$. 

Recall that in the construction of the endomorphism $\pi_{N_{0},n,z}^{\prime}$ of $H_{\star}(N_{0},(V_{\h_{n}^{\circ}\An})_{b}^{\prime})$, it factors through $H_{\star}(N_{0},(V_{\h(z)_{n}^{\circ}\An})_{b}^{\prime})$. To show that condition (1) is satisfied, it suffices to show that there is a commutative diagram \begin{equation}\label{eqn:factorization_H(z)_n,n+1}
\begin{tikzcd}
	 & \widehat{H}_{\star}(N_{0},(V_{\h_{n}^{\circ}\An})_{b}^{\prime}) \arrow[d] \\ 
D(M_{0},\m_{n}^{\circ})\underset{D(M_{0},\m_{n+1}^{\circ})}{\widehat{\otimes}} \widehat{H}_{\star}(N_{0},(V_{\h(z)_{n,n+1}^{\circ}\An})_{b}^{\prime})\arrow[r] \arrow[d] & \widehat{H}_{\star}(N_{0},(V_{\h(z)_{n}^{\circ}\An})_{b}^{\prime}) \arrow[d] \\ 
     D(M_{0},\m_{n}^{\circ})\underset{D(M_{0},\m_{n+1}^{\circ})}{\widehat{\otimes}} \widehat{H}_{\star}(N_{0},(V_{\h_{n+1}^{\circ}\An})_{b}^{\prime})\arrow[r] & \widehat{H}_{\star}(N_{0},(V_{\h_{n}^{\circ}\An})_{b}^{\prime})
\end{tikzcd}
\end{equation}
such that the middle horizontal map is an isomorphism. This is exactly the second assertion of Proposition \ref{prop:H(V_L)->H(V_H)} with the observation that the completed tensor products \begin{equation*}
    D(H(z)_{0},\h(z)_{n}^{\circ})_{N_{0}}\underset{D(H(z)_{0},\h(z)_{n,n+1}^{\circ})_{N_{0}}}{\widehat{\otimes}} \textnormal{ and }  D(M_{0},\m_{n}^{\circ})\underset{D(M_{0},\m_{n+1}^{\circ})}{\widehat{\otimes}}
\end{equation*} agree by the Iwahori decomposition of $\h(z)_{n}$ and $\h(z)_{n,n+1}$.

The left vertical and the bottom horizontal $D(M_{0},\m_{n}^{\circ})$-linear morphisms of (\ref{eqn:factorization_H(z)_n,n+1}) are defined by Proposition \ref{eqn:H(V_L)->H(V_H)}. The commutativity of the diagram follows from the fact that the rigid analytic inclusion $\h(z)_{n,n+1}\sub \h_{n}$ factors through $\h_{n+1}$ by our choice of $z$. 

It remains to show that hypothesis (2) holds. Fix $n\geq 0$ sufficiently large. Since $D(H_{0},\h_{n+1}^{\circ})\ra \dn$ is flat, we can find 
finite free resolutions $P_{n,\star}\ra (V_{\h_{u}^{\circ}\An})_{b}^{\prime}$ (resp. $P_{n+1,\star}\ra (V_{\h_{n+1}^{\circ}\An})_{b}^{\prime}$) by $D(H_{0},\h_{u}^{\circ})$-modules (resp. $D(H_{0},\h_{n+1}^{\circ})$-modules), where $P_{n,\star}$ and $P_{n+1,\star}$ have the same respective ranks. For each $k\geq0$, let $\d_{\ell,k}:\,(P_{\ell,k+1})_{N_{0}}\ra (P_{\ell,k})_{N_{0}}$ be the boundary map for each $\ell=n,n+1$ and all $k\geq 0$. Recall that the transition maps $P_{\ell,k+1}\ra P_{\ell,k}$ as well as the boundary maps $\d_{\ell,k}$ are automatically continuous and strict for each $\ell=n,n+1$ and $k\geq 0$ \cite[Prop. A.10]{Jacquet2}. 

By \cite[Prop. 4.2.22]{Jacquet1}, the continuous $D(M_{0},\m_{n}^{\circ})$-linear map \[
    D(M_{0},\m_{n}^{\circ})\underset{D(M_{0},\m_{n+1}^{\circ})}{\widehat{\otimes}}(P_{n+1,k})_{N_{0}} \ra (P_{n,k})_{N_{0}}
\] induced by the natural map $D(H_{0},\h_{n+1}^{\circ})\ra \dn$ is $D(M_{0},\m_{n}^{\circ})$-compact for each $k\geq 0$. Furthermore, \[
    D(M_{0},\m_{n}^{\circ})\underset{D(M_{0},\m_{n+1}^{\circ})}{\widehat{\otimes}}\ker(\d_{n+1,k}) \ra \ker(\d_{n,k})
\] is compact \cite[Lem. 2.3.4]{Jacquet1}. There is a natural commutative diagram \begin{equation}
    \begin{tikzcd}
	D(M_{0},\m_{n}^{\circ})\underset{D(M_{0},\m_{n+1}^{\circ})}{\widehat{\otimes}}\ker(\d_{n+1,k}) \arrow[r] \arrow[d] &  \ker(\d_{n,k}) \arrow[d] \\
	D(M_{0},\m_{n}^{\circ})\underset{D(M_{0},\m_{n+1}^{\circ})}{\widehat{\otimes}} \widehat{H}_{\star}(N_{0},(V_{\h_{n+1}^{\circ}})_{b}^{\prime})\arrow[r] & \widehat{H}_{\star}(N_{0},(V_{\h_{n}^{\circ}})_{b}^{\prime})
\end{tikzcd}.
\end{equation} Applying \cite[Lem. 2.3.4]{Jacquet1} again, we see that the bottom horizontal arrow is $D(M_{0},\m_{n}^{\circ})$-compact. Proposition $3.2.24$ of \cite{Jacquet1} completes the proof of the essential admissibility of $H^{\star}J_{P}(V)$ as a locally analytic $M$-representation.


\subsection{Delta Functor}\label{sec:delta}
In this subsection, we present the proof that $H^{\star}J_{P}(V)$ is a $\delta$-functor.

For each $n\geq 0$, let $P_{n,\star} = \dn^{r_{n,\star}} \ra \vn$ be a free resolution of $\vn$ consisting of finitely generated $\dn$-modules and let $\d_{n,k}:\,(P_{n,k})_{N_{0}}\ra (P_{n,k-1})_{N_{0}}$ be the boundary maps. For each $n\geq 0$, \[
    \widetilde{P}_{n+1,\star}:=\dn \underset{D(H_{0},\h_{n+1}^{\circ})}{\otimes} P_{n+1,\star} = \dn^{r_{n+1,\star}} \ra \vn 
\] is also a projective resolution of $\vn$ by free $\dn$-modules. The identity map on $(\vn)_{N_{0}}$ induces chain maps \begin{equation}\label{eqn:tilde_P_n+1->P_n}
    i_{n,\star}:\,(\widetilde{P}_{n+1,\star})_{N_{0}}\ra (P_{n,\star})_{N_{0}}.
\end{equation} Define the transition map $P_{n+1,\star}\ra P_{n,\star}$ to be the composition of the natural map $P_{n+1,\star}\ra \widetilde{P}_{n+1,\star}$ induced by the rigid analytic inclusion $\h_{n+1}\ra \h_{n}$ with (\ref{eqn:tilde_P_n+1->P_n}). Since the rigid analytic inclusion $\overline{\P}_{n+1}\sub\overline{\P}_{n}$ is relatively compact, the natural map $D(\overline{P}_{0},\overline{\P}_{n+1}^{\circ})\ra D(\overline{P}_{0},\overline{\P}_{n}^{\circ})$ is compact as a map between locally convex $K$-vector spaces. Then the transition maps $(P_{n+1,\star})_{N_{0}}\ra (P_{n,\star})_{N_{0}}$ are also compact, since it factors through \[
    (P_{n+1,\star})_{N_{0}} = D(\overline{P}_{0},\overline{\P}_{n+1}^{\circ})^{r_{n+1,\star}} \ra D(\overline{P}_{0},\overline{\P}_{n}^{\circ})^{r_{n+1,\star}} = (\widetilde{P}_{n+1,\star})_{N_{0}}.
\] Therefore, the projective limit $\underset{\underset{n}{\la}}{\lim}(P_{n,\star})_{N_{0}}$ is a nuclear Fr\'echet space. 

For each $\ell=n,n+1$ and $z\in Z^{+}$, let $R_{\ell,\star} = D(H(z)_{0},\h(z)_{\ell}^{\circ})^{s_{\ell,\star}} \ra (V_{H(z)_{\ell}^{\circ}\An})^\prime_{b}$ be a free resolution of $(V_{H(z)_{\ell}^{\circ}\An})^\prime_{b}$ by finitely generated $D(H_{0},\h(z)_{\ell}^{\circ})$-modules. By replacing $P_{\ell,\star}$, we can assume that $P_{\ell,\star}$ is the base change of $R_{\ell,\star}$ to $D(H_{0},\h_{\ell}^{\circ})$ over $D(H(z)_{0},\h(z)_{\ell}^{\circ})$. Then $s_{\ell,\star}\cdot[N_{0}:zN_{0}z^{-1}]=r_{\ell,\star}$. Let $\widetilde{R}_{n+1,\star}$ be the base change of $R_{n,\star}$ to $D(H(z)_{0},\h(z)_{n}^{\circ})$ over $D(H(z)_{0},\h(z)_{n+1}^{\circ})$. The identity map on $((V_{H(z)_{n}^{\circ}\An})^\prime_{b})_{N_{0}}$ induces chain maps \begin{equation}\label{eqn:tilde_R_n+1->R_n}
    j_{n,\star}:\,(R_{n,\star})_{N_{0}}\ra (\widetilde{R}_{n+1,\star})_{N_{0}}.
\end{equation} 

By Remark \ref{rmk: Hecke_on_terms}, there is an action of $z\in Z^{+}$ on each $P_{n,\star}$. However, we need to use a slightly different action of $z$. Define $\tilde{\pi}_{N_{0},n,z}^{\prime}$ to be the composition \begin{equation}\begin{split}
    (P_{n,\star})_{N_{0}}  = & \dn^{r_{n,\star}}_{N_{0}}  \overset{\res}{\longrightarrow}  \dn^{r_{n,\star}}_{zN_{0}z^{-1}} \overset{}{\cong} D(H(z^{-1})_{0},\h_{n}^{\circ})^{s_{n,\star} }_{zN_{0}z^{-1}} \\
    \overset{Prop. \ref{prop:Hecke_isom}}{\cong} & D(H(z^{-1})_{0},\h(z^{-1})_{n}^{\circ})^{s_{n,\star} }_{zN_{0}z^{-1}} \overset{z^{-1}(\cdot)z}{\cong} D(H(z)_{0},\h(z)_{n}^{\circ})^{s_{n,\star} }_{N_{0}} = (R_{n,\star})_{N_{0}} \\
    \overset{j_{n,\star}}{\ra} &  D(H(z)_{0},\h(z)_{n}^{\circ})^{s_{n+1,\star} }_{N_{0}} = (\widetilde{R}_{n,\star})_{N_{0}} \ra D(H_{0},\h_{n}^{\circ})^{r_{n+1,\star}}_{N_{0}} = (\widetilde{P}_{n+1,\star})_{N_{0}} \\
    \overset{i_{n,\star}}{\longrightarrow } & D(H_{0},\h_{n}^{\circ})^{r_{n,\star}}_{N_{0}} = (P_{n,\star})_{N_{0}}. 
\end{split}\end{equation} The only difference to Remark \ref{rmk: Hecke_on_terms} is the additional intertwining of $i_{n,\star}$ and $j_{n,\star}$, so it induces the same Hecke operator upon passing to homology. 

\begin{lemma}
For each $k\geq 0$, \[
    C^{\an}(\widehat{Z}_{M},K)\widehat{\otimes}_{K[Z^{+}]}\underset{\underset{n}{\la}}{\lim}(P_{n,k})_{N_{0}} 
\] is a coadmissible $C^{\an}(\widehat{Z}_{M},K)\widehat{\otimes}_{K}D^{\La}(M_{0},K)$-module. 
\end{lemma}
\begin{proof}
Similar to section \ref{sec:Essential_Admissibility}, we wish to show that the hypotheses of \cite[Prop. 3.2.24]{Jacquet1} hold. The $D(M_{0},\m_{n}^{\circ})$-compactness of \[
    D(M_{0},\m_{n}^{\circ}) \widehat{\otimes}_{D(M_{0},\m_{n+1}^{\circ})} (P_{n+1,k})_{N_{0}} \ra (P_{n,k})_{N_{0}}
\] follows from \cite[Lem. 2.3.4]{Jacquet1} and \cite[Prop. 4.2.22]{Jacquet1}. 

Let $z\in Z^{+}$ be an element satisfying the properties that $z^{-1}\overline{N}_{n}z\sub \overline{N}_{n+1}$ for all $n\geq 0$. By the choice of $z$, there is a commutative diagram \begin{equation}
    \begin{tikzcd}
        D(M_{0},\m_{n}^{\circ})\underset{D(M_{0},\m_{n+1}^{\circ})}{\widehat{\otimes}}(P_{n+1,\star})_{N_{0}} \arrow[r] \arrow[d] \arrow[ddd, bend right = 90, "id\widehat{\otimes} z"] & (P_{n,\star})_{N_{0}} \arrow[d] \arrow[ddd, bend left = 50, "z"] \\
        D(M_{0},\m_{n}^{\circ})\underset{D(M_{0},\m_{n+1}^{\circ})}{\widehat{\otimes}}(R_{n+1,\star})_{N_{0}} \arrow[r]  \arrow[d] &   (R_{n,\star})_{N_{0}} \arrow[d,"j_{n,\star}"] \\
        D(M_{0},\m_{n}^{\circ})\underset{D(M_{0},\m_{n+1}^{\circ})}{\widehat{\otimes}} D(H(z)_{0},\h(z)_{n,n+1}^{\circ})_{N_{0}}^{s_{n+1,\star}} \arrow[d] \arrow[r,"\varphi"] & (\widetilde{R}_{n+1,\star})_{N_{0}} \arrow[d] \\
	    D(M_{0},\m_{n}^{\circ})\underset{D(M_{0},\m_{n+1}^{\circ})}{\widehat{\otimes}} (P_{n+1,\star})_{N_{0}} \arrow[r]  & (P_{n,\star})_{N_{0}}
    \end{tikzcd}.
\end{equation}
To satisfy hypothesis (ii) of \cite[Prop. 3.2.24]{Jacquet1}, it suffices to show that $\varphi$ is an isomorphism. This follows from a consideration of the Iwahori decomposition of the relevant distribution algebras. 
\end{proof}

\begin{proposition}\label{prop:limH_H(im)}
There is a topological isomorphism \begin{equation}\label{eqn:limH_H(im)}
    \left(H^{\star}J_{P}(V)\right)^{\prime}_{b} \cong H_{\star} \left(C^{\an}(\widehat{Z}_{M},K)\widehat{\otimes}_{K[Z^{+}]}\underset{\underset{n}{\la}}{\lim}(P_{n,k})_{N_{0}}\right)
\end{equation} of coadmissible $C^{\an}(\widehat{Z}_{M},K)\widehat{\otimes}_{K}D^{\La}(M_{0},K)$-modules.
\end{proposition}
\begin{proof} First, recall that the category of coadmissible $C^{\an}(\widehat{Z}_{M},K)\widehat{\otimes}_{K}D^{\La}(M_{0},K)$-modules is abelian \cite{ST_dist}, and so the target of the isomorphism is also coadmissible. Similar to (\ref{eqn:HJp(V)_H}), it follows from \cite[Prop. 3.2.28]{Jacquet1} and \cite[Prop. 1.1.32]{analytic} that there is an isomorphism of coadmissible $C^{\an}(\widehat{Z}_{M},K)\widehat{\otimes}_{K}D^{\La}(M_{0},K)$-modules \begin{equation}\label{eqn:HJp(V)_alt}
    H_{\star}\left(C^{\an}(\widehat{Z}_{M},K)\widehat{\otimes}_{K[Z^{+}]}\underset{\underset{n}{\la}}{\lim}(P_{n,k})_{N_{0}}\right) \cong H_{\star} \left(\underset{\underset{n}{\la}}{\lim}C^{\an}(\widehat{Y}_{n},K)^{\dagger}\widehat{\otimes}_{K[u]}(P_{n,k})_{N_{0}}\right).
\end{equation} It suffices to show that the targets of (\ref{eqn:HJp(V)_H}) and (\ref{eqn:HJp(V)_alt}) agree. For each $k\geq 0$, there is a topological $K$-linear isomorphism \[
    H_{k} \left(\underset{\underset{n}{\la}}{\lim}C^{\an}(\widehat{Y}_{n},K)^{\dagger} \widehat{\otimes}_{K[C^{+}]}(P_{n,\star})_{N_{0}}\right) \cong \underset{\underset{n}{\la}}{\lim} \frac{\ker(id \widehat{\otimes} \d_{k})}{\overline{\im(id \widehat{\otimes} \d_{k+1})}},
\] where $\overline{\im(id \widehat{\otimes} \d_{k+1})}$ denotes the closure of the image of $id \widehat{\otimes} \d_{k+1}$ inside $C^{\an}(\widehat{Y}_{n},K)^{\dagger} \widehat{\otimes}_{K[u]}(P_{n,k})_{N_{0}}$ \cite[Thm. 3]{Komatsu}. In fact, a similar proof to \cite[Lem. A.11]{Jacquet2} show that $C^{\an}(\widehat{Y}_{n},K)^{\dagger} \widehat{\otimes}_{K[C^{+}]}(P_{n,\star})_{N_{0}}$ is a finitely presented module over the coherent ring $C^{\an}(\widehat{Y}_{n},K)^{\dagger} \widehat{\otimes}_{K[u]}D(M_{0},\m_{n}^{\circ})$. Additionally, the boundary maps $id \widehat{\otimes} \d_{k}$ are automatically strict with closed image \cite[Prop A.10]{Jacquet2} \cite[Cor 3, p. IV.28]{Bourbaki}. 

Since $K[u]\ra C^{\an}(\widehat{Y}_{n},K)^{\dagger}$ is flat, $C^{\an}(\widehat{Y}_{n},K)^{\dagger}\widehat{\otimes}_{K[u]} H_{\star}(N_{0},\vn)$ is exactly the completion of $\frac{\ker(id \otimes \d_{k})}{\im(id \otimes \d_{k+1})}$. The spaces $\ker(id\otimes\d_{k})$ and $\im(id\otimes\d_{k+1})$ are dense in the closed subspaces $\ker(id\widehat{\otimes}\d_{k})$ and $\im(id\widehat{\otimes}\d_{k+1})$ respectively. The proposition follows immediately.
\end{proof}

Since the definition of $\left(H^{\star}J_{P}(V)\right)^{\prime}_{b}$ is independent of the choice of resolution of $\vn$ for all $n\geq 0$, the same is true of the source of (\ref{eqn:limH_H(im)}). Suppose $0\ra U \ra V \ra W \ra 0$ is a strict short exact sequence in $\rep_{\textnormal{ad}}(G)$. For all $n\geq 0$, the sequence \begin{equation*}
    0 \ra (W_{\h_{n}^{\circ}\An})^{\prime}_{b} \ra (V_{\h_{n}^{\circ}\An})^{\prime}_{b} \ra (U_{\h_{n}^{\circ}\An})^{\prime}_{b} \ra 0
\end{equation*} is also strict exact \cite[Cor. 1.4]{ST_GL2} \cite[Thm. 1, p. IV.28]{Bourbaki} \cite[Cor. A.13]{Jacquet2}. Recall that taking $N_{0}$-coinvariants and finite slope part on the dual side are both right exact (in the strong sense that strict right exact sequences remain strict right exact) \cite[Prop. 3.2.6]{Jacquet1}. Then \cite[Thm. B]{ST_dist}, Proposition \ref{prop:limH_H(im)} and usual homological algebra show that there is a strict long exact sequence of coadmissible $C^{\an}(\widehat{Z}_{M},K)\widehat{\otimes}_{K}D^{\La}(M_{0},K)$-modules \begin{equation*}
    0 \la \left(H^{0}J_{P}(W)\right)^{\prime}_{b} \la \left(H^{0}J_{P}(V)\right)^{\prime}_{b} \la \left(H^{0}J_{P}(U)\right)^{\prime}_{b} \la \left(H^{1}J_{P}(W)\right)^{\prime}_{b} \la ....
\end{equation*} By taking the strong dual, we have just shown that $H^{\star}J_{P}$ is a $\delta$-functor. 

\begin{remark}
An alternative approach to proving that $H^{\star}J_{P}$ is a $\delta$-functor is to prove that taking finite slope part is exact. This is done in \cite[Thm. 4.5]{Weibo}.

\end{remark}


\vspace{3mm}
\section{Derived Jacquet-Emerton module of Orlik-Strauch Representations}\label{sec:example}

In this section, we present a method of computing the derived Jacquet-Emerton modules of Orlik-Strauch representations in the case $G=SL_{2}(\qp)$, along with some explicit examples. 

Let $T$ be the subgroup of $G$ consisting of diagonal matrices and let $P$ be the standard Borel subgroup of upper triangular matrices and $\overline{P}$ be the opposite Borel subgroup of lower triangular matrices. Let $w=\begin{pmatrix} 0 & 1 \\ -1 & 0 \end{pmatrix}$ be our choice of the non-trivial element of the Weyl group of $G$. Let $N_{0}$ be a maximal compact subgroup $\begin{pmatrix} 1 & \zp \\ 0 & 1 \end{pmatrix}$ of $N$. Let $z=\begin{pmatrix} p & 0 \\ 0 & p^{-1} \end{pmatrix}\in Z_{T}^{+}$, $X=\begin{pmatrix} 0 & 1 \\ 0 & 0 \end{pmatrix}\in \N$, $Y=\begin{pmatrix} 0 & 0 \\ 1 & 0 \end{pmatrix}\in \overline{\N}$ and $H=\begin{pmatrix} 1 & 0 \\ 0 & -1 \end{pmatrix}\in \mathfrak{t}$. For a representation $\sigma$ of $\overline{P}$, define \[
    \ind_{\overline{P}}^{G}\sigma = \left\{ f\in C^{\La}(G,\sigma)\mid f(pg)=\sigma(p)f(g)\, \forall\, p\in \overline{P}, \textnormal{ and }g\in G\right\}.
\]

Let $M$ be an object in the category $\o^{\overline{\mathfrak{p}}}$ as defined in \cite{OS_Jordan_Holder}. Suppose $\psi$ is a smooth character of $T$. Let $W$ be a $U(\overline{\mathfrak{p}})$-invariant $K$-subspace of $M$ that generates $M$ as a $U(\g)$-module. Then there is a short exact sequence of $U(\g)$-modules \[
    0 \ra \mathfrak{d} \ra U(\g)\otimes_{U(\overline{\mathfrak{p}})}W \ra M \ra 0.
\] Define a left action of $\g$ on $C^{\La}(G,W^\prime\otimes\psi)$ by \begin{align*}
    (\tau\cdot f)(g)= \frac{d}{dt} \mid_{t=0}f\left(\exp(-t\tau)g\right),
\end{align*} which extends to an action of $U(\g)$. Each $\tau\otimes w\in U(\g)\otimes_{K}W$ defines a $K$-linear continuous map \begin{align*}
    C^{\La}(G,W^\prime\otimes \psi) & \ra C^{\La}(G,\psi) \\
    f & \map  (\tau\otimes w)\cdot f :\, g\map (\tau\cdot f)(g)(w). 
\end{align*} As in \cite{OS_Jordan_Holder}, the Orlich-Strauch representation $\mathcal{F}_{\overline{P}}^{G}(M,\psi)$ is defined as the subspace of $\ind_{\overline{P}}^{G}W^\prime \otimes \psi$ consisting of functions $f$ such that $\mathfrak{s}\cdot f=0$ for all $\mathfrak{s}\in \mathfrak{d}$. 

For each $k\in 2\z$, let $\chi_{k}$ be the algebraic character of $T$ of weight $k$. Let $M(k)$ be the Verma module $U(\g)\otimes_{U(\overline{\mathfrak{p}})}\chi_{k}$ associated to $\chi_{k}$ and $L(k)$ its simple quotient. Then $\mathcal{F}_{\overline{P}}^{G}(M(-k),\psi)$ is the locally analytic principal series representation $\ind_{\overline{P}}^{G}\chi_{k}\psi$ and for $k\geq 0$, $\mathcal{F}_{\overline{P}}^{G}(L(-k),\psi)$ is the locally algebraic representation $L(-k)\otimes_{K}\textnormal{sm-Ind}_{\overline{P}}^{G}\psi$. 

In fact, for each $k\geq 0$, by applying the contravariant exact functor $\mathcal{F}_{\overline{P}}^{G}(\cdot ,\psi)$ to the BGG resolution \[
    0 \ra M(k+2) \ra M(-k) \ra L(-k) \ra 0,
\] there is a short exact sequence \begin{equation}\label{eq:loc_alg->ind->neg_wt}
    0 \ra L(-k)\otimes_{K}\textnormal{sm-Ind}_{\overline{P}}^{G}\psi \ra \ind_{\overline{P}}^{G}\chi_{k}\psi \ra \ind_{\overline{P}}^{G}\chi_{-k-2}\psi \ra 0
\end{equation} (see also \cite{ST_U(g)}). We now compute the derived Jacquet modules of each term of this exact sequence. 

Evaluation at $w$ induces a $P$-invariant surjection \[ \mathcal{F}_{\overline{P}}^{G}(M,\psi) \surj \mathcal{F}_{\overline{P}}^{G}(M,\psi)_{w}, \] where $\mathcal{F}_{\overline{P}}^{G}(M,\psi)_{w}$ is the stalk of $\mathcal{F}_{\overline{P}}^{G}(M,\psi)$ at $w$ \cite[Def. 2.4.2]{Jacquet2}. The kernel of this map is the subspace $\mathcal{F}_{\overline{P}}^{G}(M,\psi)(N)$ of $\mathcal{F}_{\overline{P}}^{G}(M,\psi)$ consisting of locally analytic functions with support in $\overline{P}N$. Taking $\N$-cohomology followed by $N_{0}$-invariants (keeping in mind the exactness of $N_{0}$-invarince on smooth representations) produces a long exact sequence \begin{align}\begin{split}\label{eq:n_cohom1}
    0 & \ra H^{0}(\N,\mathcal{F}_{\overline{P}}^{G}(M,\psi)(N))^{N_{0}} \ra H^{0}(\N,\mathcal{F}_{\overline{P}}^{G}(M,\psi))^{N_{0}} \ra H^{0}(\N,\mathcal{F}_{\overline{P}}^{G}(M,\psi)_{w})^{N_{0}} \\ 
    & \ra H^{1}(\N,\mathcal{F}_{\overline{P}}^{G}(M,\psi)(N))^{N_{0}} \ra H^{1}(\N,\mathcal{F}_{\overline{P}}^{G}(M,\psi))^{N_{0}} \ra H^{1}(\N,\mathcal{F}_{\overline{P}}^{G}(M,\psi)_{w})^{N_{0}} \ra 0.
\end{split}\end{align} 

Under the canonical identification $N= \begin{pmatrix} 1 & \qp \\ 0 & 1 \end{pmatrix}\cong \qp$, the space $\mathcal{F}_{\overline{P}}^{G}(M,\psi)(N)$ can be viewed as the subspace of $C^{\La}_{c}(\qp,W^{\prime}\otimes\psi)$ on which $\mathfrak{d}\cdot f=0$, which we denote by $C^{\La}_{c}(\qp,W^{\prime}\otimes\psi)^{\mathfrak{d}}$ \cite[Lem. 2.3.3]{Jacquet2}. The action of $X$ on $\mathcal{F}_{\overline{P}}^{G}(M,\psi)(N)$ is the derivative map $\frac{d}{dx}$. 

\begin{lemma}\label{lem:n_cohom_N_section}
For each $i$, the cohomology groups $H^{i}(\N,\mathcal{F}_{\overline{P}}^{G}(M,\psi)(N))$ and $H^{i}(\N,\mathcal{F}_{\overline{P}}^{G}(M,\psi)^{\textnormal{lp}}(N))$ coincide.
\end{lemma}
\begin{proof}
Since $\N$ is one dimensional, it suffices to check this for $i=0,1$. The action of $\N$ is differentiation, so the $0$-th cohomology consists of locally constant functions, which are locally polynomials. 

For a $\N$-module $V$, let $V_{\N}$ denote the space of $\N$-coinvariants $V/\N V$. Since $\N$ is one dimensional, \begin{equation}\label{eq:n_cohom_formula}
    H^{1}(\N,V)\cong V_{\N}\otimes \N^{\star}.
\end{equation} To show that the first $\N$-cohomology coincides, it suffices to show that the natural map \begin{equation*}
    \mathcal{F}_{\overline{P}}^{G}(M,\psi)^{\textnormal{lp}}(N))_{\N} \inj \mathcal{F}_{\overline{P}}^{G}(M,\psi)(N)_{\N}
\end{equation*} is an isomorphism. Injectivity is clear, as the anti-derivative of a polynomial is a polynomial. For surjectivity, it suffices to show that any element of $(C^{\La}_{c}(\qp,W^{\prime}\otimes\psi)^{\mathfrak{d}})_{\N}$ has a locally polynomial representative.

Suppose $f$ is an analytic function from a compact open subset $\Omega\sub N$ to $K$. Then there is a finite compact open cover $(\Omega_{i})_{i=1}^{\ell}$ of $\Omega$ on which the formal integral of $f$ converges on each $\Omega_{i}$. This shows that \[
    C^{\La}_{c}(\qp,W^{\prime}\otimes\psi)_{\N} := \frac{C^{\La}_{c}(\qp,W^{\prime}\otimes\psi)}{\N C^{\La}_{c}(\qp,W^{\prime}\otimes\psi)} = 0.
\] 

Now suppose $f\in C^{\La}_{c}(\qp,W^{\prime}\otimes\psi)^{\mathfrak{d}}$, and suppose $h\in C^{\La}_{c}(\qp,W^{\prime}\otimes\psi)$ is the formal integral of $f$, that is $f=Xh$. Assume $f$ is non-trivial in $(C^{\La}_{c}(\qp,W^{\prime}\otimes\psi)^{\mathfrak{d}})_{\N}$, that is, $\mathfrak{d}\cdot h\neq 0$. For a vector $v\in W$, let $\widehat{v}^{\perp}$ be the kernel of the evaluation at $v$ map on $W^{\prime}$. Let $\widehat{v}$ be a basis element of the complement of $\widehat{v}^{\perp}$, so that $W^{\prime}=\widehat{v}^{\perp}\oplus K\widehat{v}$. The projections of $W^{\prime}$ onto $K\widehat{v}$ and $\widehat{v}^{\perp}$ induces projection of $C^{\La}_{c}(\qp,W^{\prime}\otimes\psi)$ onto $C^{\La}_{c}(\qp,K\widehat{v}\otimes \psi)$ and $C^{\La}_{c}(\qp,\widehat{v}^\perp\otimes\psi)$ respectively. Let $f_{\widehat{v}}$ and $f_{\widehat{v}^{\perp}}$ be the image of $f$ under the respective projections, and similarly for $h_{\widehat{v}}$ and $h_{\widehat{v}^{\perp}}$. 

By definition, the action of $X^{n}\otimes v \in\mathfrak{d}\sub U(\N)\otimes_{K}W$ acts on $f$ by the formula \[
    ((X^{n}\otimes v)\cdot f )(x) = (-1)^{n}\left(\frac{d^n f}{dx^{n}}(x)\right)(v).
\] It is easy to see that $(X^{n}\otimes v)\cdot f_{\widehat{v}^{\perp}} $ and $(X^{n}\otimes v)\cdot h_{\widehat{v}^{\perp}}$ are zero. By assumption, $(X^{n}\otimes v)\cdot f_{\widehat{v}}=0$ and $(X^{n}\otimes v)\cdot h_{\widehat{v}} \neq 0$. Since $f_{\widehat{v}}$ is simply the derivative of $h_{\widehat{v}}$, we see that $(1\otimes v)f_{\widehat{v}}$ must be a polynomial of degree $n-1$ in $C^{\La}_{c}(\qp,K\widehat{v}\otimes \psi)$. 

Since $U(\g)$ is Noetherian, $M$ is finitely generated (and hence finitely presented), and so $\mathfrak{d}$ is also finitely generated. By repeating the above argument on the generators of $\mathfrak{d}$, we can decompose $f=f_1+f_2$ where $f_1\in  C^{\textnormal{lp}}(\qp,W^{\prime}\otimes\psi)^{\mathfrak{d}}$ and $f_{2}\in \N  C^{\La}_{c}(\qp,W^{\prime}\otimes\psi)^{\mathfrak{d}}$. This shows that $f$ has a locally polynomial representative in $C^{\textnormal{lp}}(\qp,W^{\prime}\otimes\psi)^{\mathfrak{d}}_{\N}$ as required. 
\end{proof}

We prove a similar result for the stalk at $w$. By \cite[Lem. 2.3.5]{Jacquet2}, right translation by $w$ induces a topological isomorphism \begin{equation}\label{eq:change_of_section}
    (\ind_{\overline{P}}^{G}W^{\prime}\otimes\psi)_{w} \cong (\ind_{\overline{P}}^{G}W^{\prime}\otimes\psi)_{e}
\end{equation} where $e$ is the identity element of $N$. The isomorphism interpolates the $(\g,P)$-action on the source with the $(\textnormal{Ad}_{w}(\g),\overline{P})$-action on the target. Therefore, there is a $K$-linear isomorphism \[
    H^{i}(\N,(\ind_{\overline{P}}^{G}W^{\prime}\otimes\psi)_{w}) \cong H^{i}(\overline{\N},\ind_{\overline{P}}^{G}W^{\prime}\otimes\psi)_{e})
\] for all $i$, and similarly for the cohomology groups of $\mathcal{F}_{\overline{P}}^{G}(M,\psi)_{w}$. 

\begin{lemma}
For each $i$, the cohomology groups $H^{i}(\overline{\N},\mathcal{F}_{\overline{P}}^{G}(M,\psi)_{e})$ and $H^{i}(\overline{\N},\mathcal{F}_{\overline{P}}^{G}(M,\psi)^{\textnormal{pol}}_{e})$ coincide.
\end{lemma}
\begin{proof}
Suppose $\Omega\sub N$ is a compact open neighbourhood of $e$. Since \[
    \begin{pmatrix}
    1 & x \\ 0 & 1
    \end{pmatrix} \begin{pmatrix}
    1 & 0 \\ t & 1
    \end{pmatrix} = \begin{pmatrix}
    1+xt & x \\ t & 1
    \end{pmatrix} = \begin{pmatrix}
    1+xt & 0 \\ 0 & (1+xt)^{-1}
    \end{pmatrix} \begin{pmatrix}
    1 & 0 \\ t(1+xt) & 1
    \end{pmatrix} \begin{pmatrix}
    1 & \frac{x}{1+xt} \\ 0 & 1
    \end{pmatrix}, 
\] under the identification $N\cong \qp$, the action of $Y$ on $C^{\La}(\Omega,W^{\prime}_{i}\otimes\psi)$ is given by the formula \begin{equation}\label{eq:Y_action_formula}
    (Yf)(x) = xHf(x) + Yf(x) - x^{2}f^{\prime}(x). 
\end{equation} 
By the definition of $M$, the action of $H$ on $M$ is semisimple and algebraic. Additionally, there is a Jordan-Holder series in $\o_{\alg}^{\overline{\mathfrak{p}}}$ \[
    M= M_{r} \supsetneq M_{r-1} \supsetneq ... M_{0} = 0,
\] where $YM_{i} \sub M_{i-1}$ for all $i\geq 1$. For each $i$, let $W_{i}=W\cap M_{i}$. This induces a filtration on $W^{\prime}$, where $YW^{\prime}_{i}\sub W^{\prime}_{i+1}$ for all $i\geq 1$, with $W^{\prime}_{r+1}$ taken to be zero. In fact, $Y\ker(W_{i}^{\prime}\ra W_{i-1}^{\prime}) \sub \ker(W_{i+1}^{\prime}\ra W_{i}^{\prime})$.

Let $p(x)\in C^{\La}(\Omega,\psi)$ and $v\in W^{\prime}_{i}$. Suppose $Hv=cv$ for some $c\in \z$. Then the action of $Y$ on $p(x)\otimes v\in C^{\La}(\Omega,\psi)\otimes_{K}W^{\prime}_{i}= C^{\La}(\Omega,W^{\prime}_{i}\otimes\psi) $ is given by \[
    Y(p(x)\otimes v)= (cxp(x)-x^2p^{\prime}(x))\otimes v + p(x)\otimes Yv.
\] This is zero if and only if $p(x)=x^c$ and $v\in \ker(W^{\prime}\ra W_{r-1}^{\prime})$. This holds true for all $\Omega$, thus \begin{align*}
    C^{\La}(N,W^{\prime}\otimes\psi)^{Y}_{e} &= C^{\textnormal{pol}}(N,W^{\prime}\otimes\psi)^{Y}_{e}.
\end{align*} The same holds true when restricting to the subspace of elements annihilated by $\mathfrak{d}$, which proves the claim for the $0$-th cohomology group. 

\begin{claim}
Let $\Omega$ be a compact open neighbourhood of $N$ and for each $i$ let $C_{i}=\ker(W^{\prime}\ra W_{i-1}^{\prime})$. Then for all $1\leq i \leq r$, there is a positive integer $n_{i}$ such that $p(x)\otimes v \in C^{\La}(\Omega,C_{i}\otimes\psi)$ is in $YC^{\La}(\Omega,C_{i}\otimes\psi)$ if all monomials of $p(x)$ are of degree greater than $n_{i}$. 
\end{claim}
\begin{proof}
We induct backwards on $i$. Starting from $i=r$, since $Y\ker(W^{\prime}\ra W_{r-1}^{\prime})=0$, \[
    Y(x^n\otimes v) = (c_{r}-n)x^{n+1}\otimes v,
\] where $v\in C_{r}$ and $Hv=c_{r}$. Therefore, $YC^{\La}(\Omega,C_{r}\otimes\psi)$ is the complement of the span of the monomials $x^{c+1}\otimes v$ and $x^{0}\otimes v$. This establishes the claim in the base case. 

Assume the claim is true for $i$. Suppose $n>n_{i}$, $v\in C_{i-1}$ and $Hv=c_{r-1}v$, then \[
    Y(x^n\otimes v) = (c_{r-1}-n)x^{n+1}\otimes v + x^{n}\otimes Yv.
\] By assumption, $x^{n}\otimes Yv\in C^{\La}(\Omega,C_{i}\otimes\psi)$. Hence, $x^{n}\otimes v\in C^{\La}(\Omega,C_{i-1}\otimes\psi)$ for all $n > \max(n_{i},c_{r-1}+1)$ and the claim follows. 
\end{proof}

Applying the claim to $i=1$, with $C_{1}=W^{\prime}$, we deduce that every element of $C^{\La}(\Omega,W^{\prime}\otimes\psi)_{\overline{\N}}$ has a coset representative in $C^{\textnormal{pol}}(\Omega,W^{\prime}\otimes\psi)$. The same is true for $(C^{\La}(N,W^{\prime}\otimes\psi)_{e})_{\overline{\N}}$. A similar argument as the one used at end of the proof of Lemma \ref{lem:n_cohom_N_section} allow us to deduce the same result on the subspace of elements annihilated by $\mathfrak{d}$. That is, every element of $\mathcal{F}_{\overline{P}}^{G}(M,\psi)_{e}$ has a coset representative in $\mathcal{F}_{\overline{P}}^{G}(M,\psi)^{\textnormal{pol}}_{e}$. This completes the proof of the lemma. 
\end{proof}

Define an action of $U(\g)$ on $\hom_{K}(M,K)$ via the formula \[
    (\tau \cdot f)(m)=f(\dot{\tau}m) \textnormal{ for all } \tau \in U(\g) \textnormal{ and } m\in M,
\] where $\tau\map \dot{\tau}$ is the involution on $U(\g)$ induced by the multiplication by $-1$ map on $\g$. Let $\hom_{K}(M,K)^{\N^{\infty}}$ denote the subspace of $\hom_{K}(M,K)$ consisting of functions annihilated by a finite power of $\N$.  From the proof of \cite[Prop. 4.2]{Breuil_socle2}, there are $\g$-equivariant isomorphisms \begin{align}\begin{split}\label{eq:Breuil_N-functions}
    \mathcal{F}_{\overline{P}}^{G}(M,\psi)^{\textnormal{lp}}(N) &= \hom_{K}(M,K)^{\N^{\infty}}\otimes_{K}C_{c}^{\text{sm}}(N,\psi) \textnormal{ and} \\
    \mathcal{F}_{\overline{P}}^{G}(M,\psi)^{\textnormal{pol}}_{e} &\cong \hom_{K}(M,K)^{\N^{\infty}}\otimes_{K}\psi.
\end{split}\end{align} 
For each $m\in M$, let $\widehat{m}$ denote the dual of $m$. Since $m$ is $\overline{\N}$-finite, each $\widehat{m}$ is $\N$-finite and are elements of $\hom_{K}(M,K)^{\N^{\infty}}$. It is clear that these dual elements generate the space $\hom_{K}(M,K)^{\N^{\infty}}$. Computing the $\N$-cohomology of this space is then much simpler, and makes no references to $W$. This result also holds when $G=GL_{2}(\qp)$. 

In summary, we can compute the derived Jacquet-Emerton modules of Orlich-Strauch representations through their locally polynomial parts. This can be seen as a generalization of the role that $I_{\overline{P}}^{G}$ plays for $J_P$ \cite{Jacquet2}. A similar result should also hold for general $p$-adic reductive groups. 

We finish this section with some explicit calculations of the derived Jacquet-Emerton modules of the Orlik-Strauch representations introduced at the beginning of this section.

\begin{example}\label{ex:Jacquet_indchi} We first consider the case $M=M(-k)$. Let $e_{0}$ be a basis vector of $K(\chi_{-k})$ and let $e_{i}=X^{i}e_{0}$. For each $i$, let $\widehat{e}_{i}$ be the dual basis vector to $e_{i}$. A simple computation shows that each $\widehat{e}_{i}$ has weight $k-2i$ and \[
    X\widehat{e}_{i} = -\widehat{e}_{i-1}
\] with $\widehat{e}_{-1}$ taken to be zero. Then \begin{align*}
    H^{0}(\N,(\ind_{\overline{P}}^{G}\chi_{k}\psi)(N)) &\cong \widehat{e}_{0} \otimes C_{c}^{\text{sm}}(N,\psi) \text{ and}\\
    H^{1}(\N,(\ind_{\overline{P}}^{G}\chi_{k}\psi)(N)) &\cong 0.
\end{align*} The action of $N_{0}$ is via translation by $\zp$. The $N_{0}$-invariant element of $C_{c}^{\text{sm}}(N,\psi)$ is exactly the constant $K(\psi)$-valued functions on $\zp$, where $Z_{T}^{+}$ acts via the character $\psi\delta_{P}$. Therefore, as $Z_{T}^{+}$-representations \begin{align*}
    H^{0}(\N,(\ind_{\overline{P}}^{G}\chi_{k}\psi)(N))^{N_{0}} &\cong \chi_{k}\psi\delta_{P} \textnormal{ and}\\
    H^{1}(\N,(\ind_{\overline{P}}^{G}\chi_{k}\psi)(N))^{N_{0}} &\cong 0.
\end{align*}

On the other hand, since $\left[Y,X^{i}\right]=i(i-1)X^{i-1}-iHX^{i-1}$, \begin{align*}
    Ye_{i} &= i(i-1)e_{i-1}-i(-k+2(i-1))e_{i-1} = i(k-(i-1))e_{i-1} \textnormal{ and} \\
    Y\widehat{e}_{i} &= -(i+1)(k-i))\widehat{e}_{i+1}.
\end{align*} Therefore, \begin{align*}
    H^{0}(\overline{\N},(\ind_{\overline{P}}^{G}\chi_{k}\psi)_{e}) &\cong \begin{cases}
    \widehat{e}_{k}\otimes_{K}\psi & \text{if } k \geq 0 \\
    0 & \text{else}
    \end{cases},  \text{ and}\\
    H^{1}(\overline{\N},(\ind_{\overline{P}}^{G}\chi_{k}\psi)_{e}) &\cong \begin{cases}
    (\widehat{e}_{0}\oplus \widehat{e}_{k+1})\otimes_{K}\psi & \text{if } k \geq 0 \\
    \widehat{e}_{0}\otimes_{K}\psi & \text{else}
    \end{cases}. 
\end{align*} Since $\overline{\N}$ is unipotent, the exponential map $\exp: \overline{\N}\ra\overline{N}$ is an isomorphism. The above cohomology groups are finite dimensional, so the action of $\overline{\N}$ on these groups comes from the exponential of the $\overline{\N}$-action. The lie algebra $\overline{\N}$ annihilates these spaces, so the $\overline{N}$-action must be trivial. In particular, $\overline{N}_{0}=wN_{0}w^{-1}$ must also act trivially on this space. To summarize, taking into account of the interpolation of the $P$-action with the $\overline{P}$-action of (\ref{eq:change_of_section}) as well as (\ref{eq:n_cohom_formula}), we find that as $T$-representations \begin{align*}
    H^{0}(\overline{\N},(\ind_{\overline{P}}^{G}\chi_{k}\psi)_{e})^{\overline{N}_{0}} &\cong \begin{cases}
    \chi_{k}\psi^{w} & \text{if } k \geq 0 \\
    0 & \text{else}
    \end{cases}, \text{ and} \\
    H^{1}(\overline{\N},(\ind_{\overline{P}}^{G}\chi_{k}\psi)_{e})^{\overline{N}_{0}} &\cong \begin{cases}
    (\chi_{-(k+2)}\oplus \chi_{k})\otimes_{K}\psi^{w} & \text{if } k \geq 0 \\
    \chi_{-(k+2)}\psi^{w} & \text{else}
    \end{cases}. 
\end{align*} Here, $\psi^{w}$ is the smooth character of $T$ defined by $\psi^{w}(x)=\psi(w^{-1}xw)$.

From (\ref{eq:n_cohom1}), we deduce that as $Z_{T}^{+}$-representations: \begin{equation}\label{eq:DJp_Indchi}
    H^{i}(\N,(\ind_{\overline{P}}^{G}\chi_{k}\psi))^{\overline{N}_{0}} \cong \begin{cases}
    \textnormal{a class in } \Ext^{1}(\chi_{k}\psi^{w},\chi_{k}\psi\delta_{P}) & \text{if } k \geq 0, i=0 \\
    (\chi_{-(k+2)}\oplus \chi_{k})\otimes_{K}\psi^{w} & \text{if } k \geq 0, i=1  \\
    \chi_{k}\psi\delta_{P} & \text{if } k < 0, i=0 \\
    \chi_{-(k+2)}\psi^{w} & \text{if } k < 0, i=1 
    \end{cases}.  
\end{equation} In all of these cases, the action of $Z_{T}^{+}$ is invertible. Therefore, they coincide with their finite-slope parts. Hence, these are also the derived Jacquet-Emerton modules $H^{i}(J_{P}(\ind_{\overline{P}}^{G}\chi_{k}\psi))$. The reader can compare this result to the computation of section 5 of \cite{Jacquet2}.
\end{example} 

\begin{example}\label{ex:L(k)} Suppose $k\geq 0$. In the beginning of this section, we see that $\mathcal{F}_{\overline{P}}^{G}(L(-k),\psi)\cong L(-k)\otimes_{K}\textnormal{sm-Ind}_{\overline{P}}^{G}\psi$. We can compute its derived Jacquet-Emerton modules in the same manner as above with $M=L(-k)$. Here is an alternative approach. Since $\mathcal{F}_{\overline{P}}^{G}(M,\psi)$ coincides with the representation $I_{\overline{P}}^{G}(\chi_{k}\psi)$ defined in \cite{Jacquet2}, $J_{P}(I_{\overline{P}}^{G}(\chi_{k}\psi))$ is equal to $J_{P}(\ind_{\overline{P}}^{G}\chi_{k}\psi)$, which by the previous example, is an extension of $\chi_{k}\psi^{w}$ by $\chi_{k}\psi\delta_{P}$ (or see example 5.1.9 of \cite{Jacquet2}). Kostant's theorem \cite[VI Thm 6.12]{Knapp} implies that $H^{1}(\N,\textnormal{sm-}\ind_{\overline{P}}^{G}\psi))^{N_{0}}$ is isomorphic to $(\chi_{-(k+2)} \otimes_{K} \textnormal{sm-} \ind_{\overline{P}}^{G} \psi)^{N_{0}}$. Theorem \ref{thm:lie_cohom}, Proposition  4.3.5 of \cite{Jacquet2} and Theorem 6.3.5 of \cite{Casselman} combine to imply that $H^{1}J_{P}(\textnormal{sm-}\ind_{\overline{P}}^{G}\psi))\cong \chi_{\lambda^{\prime}} \otimes_{K} (\textnormal{sm-} \ind_{\overline{P}}^{G} \psi)_{N}$ has Jordan-Holder factors $\chi_{-(k+2)}\psi\delta_{P}$ and $\chi_{-(k+2)}\psi^{w}$. This result also agrees with the long exact sequence obtained from applying the Jacquet-Emerton functor to the short exact sequence (\ref{eq:loc_alg->ind->neg_wt}) and applying the results of the previous example.
\end{example}

\begin{example}
Suppose $k\geq 0$. In this example, we compute the derived Jacquet-Emerton modules of $\mathcal{F}_{\overline{P}}^{G}(M(-k)^{\vee},\psi)$, where $(\cdot)^{\vee}$ denotes the duality functor in category $\o$ (see \cite{Humphreys}). 

We can describe $M(-k)^{\vee}$ in the following way. It has a basis $e_{0},e_{1},...$ where $e_{i}$ has weight $-k+2i$, $Ye_{i}=e_{i-1}$ and $Xe_{i}=(i+1)(-k+i)e_{i+1}$. For each $i$, let $\widehat{e}_{i}$ be the dual basis of $e_{i}$. By definition, $\widehat{e}_{i}$ has weight $k-2i$, $X\widehat{e}_{i}=-i(-k+(i-1))\widehat{e}_{i-1}$ and $Y\widehat{e}_{i}=-\widehat{e}_{i+1}$. A simple computation shows that
\begin{align*}
    H^{0}(\N,\mathcal{F}_{\overline{P}}^{G}(M(-k)^{\vee},\psi)(N)) &\cong (\widehat{e}_{0}\oplus\widehat{e}_{k+1}) \otimes C_{c}^{\text{sm}}(N,\psi) \textnormal{ and} \\
    H^{1}(\N,\mathcal{F}_{\overline{P}}^{G}(M(-k)^{\vee},\psi)(N)) &\cong \widehat{e}_{k} \otimes C_{c}^{\text{sm}}(N,\psi).
\end{align*} The $N_{0}$-invariant elements are simply the constant functions on $N_{0}$. Taking (\ref{eq:n_cohom_formula}) into account, as $Z_{T}^{+}$-representations \begin{align*}
    H^{0}(\N,\mathcal{F}_{\overline{P}}^{G}(M(-k)^{\vee},\psi)(N))^{N_{0}} &\cong (\chi_{k}\oplus\chi_{-(k+2)}) \psi\delta_{P} \textnormal{ and} \\
    H^{1}(\N,\mathcal{F}_{\overline{P}}^{G}(M(-k)^{\vee},\psi)(N))^{N_{0}} &\cong \chi_{-(k+2)} \psi\delta_{P}.
\end{align*}

Furthermore, \begin{align*}
    H^{0}(\overline{\N},\mathcal{F}_{\overline{P}}^{G}(M(-k)^{\vee},\psi)_{e}) &= 0, \textnormal{ and} \\
    H^{1}(\overline{\N},\mathcal{F}_{\overline{P}}^{G}(M(-k)^{\vee},\psi)_{e}) &\cong \widehat{e}_{0} \psi.
\end{align*} By the same argument as in example \ref{ex:Jacquet_indchi}, the group $\overline{N}_{0}$ acts trivially. Taking into account of (\ref{eq:change_of_section}) and (\ref{eq:n_cohom_formula}), as $T$-representations \begin{align*}
    H^{0}(\N,\mathcal{F}_{\overline{P}}^{G}(M(-k)^{\vee},\psi)_{w})^{N_{0}} &= 0 \textnormal{ and} \\
    H^{1}(\N,\mathcal{F}_{\overline{P}}^{G}(M(-k)^{\vee},\psi)_{w})^{N_{0}} &\cong \chi_{-(k+2)}\psi^{w}.
\end{align*}

From (\ref{eq:n_cohom1}), we deduce that as $Z_{T}^{+}$-representations: \begin{align}\begin{split}
    H^{0}(\N,\mathcal{F}_{\overline{P}}^{G}(M(-k)^{\vee},\psi))^{\overline{N}_{0}} & \cong 
    (\chi_{k}\oplus\chi_{-(k+2)}) \psi\delta_{P} \textnormal{ and} \\
    H^{1}(\N,\mathcal{F}_{\overline{P}}^{G}(M(-k)^{\vee},\psi))^{\overline{N}_{0}} &\in \Ext^{1}(\chi_{-(k+2)}\psi^{w},\chi_{-(k+2)} \psi\delta_{P}).  
\end{split}\end{align} In all of these cases, the action of $Z_{T}^{+}$ is invertible. Therefore, they coincide with their finite-slope parts. Hence, these are also the derived Jacquet-Emerton modules $H^{i}(J_{P}(\mathcal{F}_{\overline{P}}^{G}(M(-k)^{\vee},\psi)))$. 

\end{example}
\begin{remark}
There is a short exact sequence \[
    0 \ra L(-k) \ra M(-k)^{\vee} \ra M(k+2) \ra 0.
\] Applying the Orlik-Strauch functor $\mathcal{F}_{\overline{P}}^{G}(\cdot,\psi)$, we deduce that there is a short exact sequence \[
    0\ra \ind_{\overline{P}}^{G}\chi_{-(k+2)}\psi \ra \mathcal{F}_{\overline{P}}^{G}(M(-k)^{\vee},\psi) \ra L(-k)\otimes_{K}\textnormal{sm-Ind}_{\overline{P}}^{G}\psi \ra 0.
\] Applying the Jacquet-Emerton functor gives us a long exact sequence of $T$-modules, which is consistent with the results of the examples above.
\end{remark}


\vspace{3mm}
\section{Applications and Conjectures}\label{sec:applications}

We use the notation established at the end of section \ref{sec:intro}. Let $U$ be an allowable object of $\rep_{\text{la,c}}^{z}(M)$ and let $V$ be a very strongly admissible $G$-representation. Theorem 0.13 of \cite{Jacquet2} states that there is an isomorphism \begin{equation}\label{eq:Em_adjunction}
    \l_{G}(I_{\overline{P}}^{G}(U),V)\cong \l_{M}(U(\delta_{P}),J_{P}(V))^{\bal}.
\end{equation} From this adjunction formula, one may expect that there is a spectral sequence of the form \[
    \Ext^{i}_{M}(U(\delta_{P}),R^{j}J_{P}(V)) \implies \Ext^{i+j}_{G}(I_{\overline{P}}^{G}(U),V),
\] especially in the case where $M$-equivariant linear maps from $U(\delta_{P})$ to $J_{P}(V)$ are automatically balanced. Here, $R^{j}J_{P}(\cdot)$ is the universal derived Jacquet-Emerton functor. The edge morphisms associated to the spectral sequence give rise to an exact sequence \begin{align}\begin{split}\label{eq:edge_morph}
0 & \ra \Ext^{1}_{M}(U(\delta_{P}),J_{P}(V)) \ra \Ext^{1}_{G}(I_{\overline{P}}^{G}(U),V) \ra \l_{M}(U(\delta_{P}),R^{1}J_{P}(V)) \\
& \ra \Ext^{2}_{M}(U(\delta_{P}),V) \ra \Ext^{2}_{G}(I_{\overline{P}}^{G}(U),V) \ra ...
\end{split}.\end{align} 

We do not know if the $\delta$-functor $H^{\star}J_P$ is universal; however, there is an injection of natural transformations $R^{\star}J_P\inj H^\star J_{P}$. Thus, if we replace $R^\star J_P$ with $H^\star J_P$ in (\ref{eq:edge_morph}), then only the first row will remain exact. We conjecture that under certain suitable assumptions, there is a short exact sequence of the form \begin{align}\label{eq:edge_morph2}
0 \ra \Ext^{1}_{M}(U(\delta_{P}),J_{P}(V)) \overset{\alpha}{\ra} \Ext^{1}_{G}(I_{\overline{P}}^{G}(U),V) \overset{\beta}{\ra} \l_{M}(U(\delta_{P}),H^{1}J_{P}(V)).
\end{align} Some similar exact sequences could potentially be derived from the adjunction formulas of \cite[Rem. 4.4]{Breuil_socle2} or \cite[Thm. A, Thm. B]{Bergall_Chojecki} as well.

We close this section with some results on various weaker versions of (\ref{eq:edge_morph2}). The common theme is that we need some additional assumptions to ensure certain maps between allowable objects of $\rep_{\text{la,c}}^{z}(M)$ are balanced. 

\begin{proposition}\label{prop:inj_beta}
Suppose $\Ext^{1}_{M}(U(\delta_{P}),J_{P}(V))=0$, $J_P(I_{\overline{P}}^{G}(U))=U(\delta_{P})$ and $U(\g)\otimes_{U(\mathfrak{p})} U$ is irreducible. Then applying the Jacquet-Emerton functor $J_P$ induces an injective map \[
    \Ext^{1}_{G}(I_{\overline{P}}^{G}(U),V) \inj  \l_{M}(U(\delta_{P}),H^{1}J_{P}(V)).
\]
\end{proposition}
\begin{proof}
Suppose $[C]$ is a class in $\Ext^{1}_{G}(I_{\overline{P}}^{G}(U),V)$ that is trivial in $\l_{M}(U(\delta_{P}),H^{1}J_{P}(V))$. By assumption, there is a short exact sequence \[
    0 \ra J_P(V) \ra J_P(C) \ra U(\delta_P) \ra 0
\] that splits. Since $U(\g)\otimes_{U(\mathfrak{p})} U$ is assumed to be irreducible, the splitting $U(\delta_P)\ra J_P(C)$ is trivially balanced. The adjunction formula (\ref{eq:Em_adjunction}) gives us a $G$-equivariant linear map $I_{\overline{P}}^{G}(U) \ra C$, which splits the class $[C]$.
\end{proof}

\begin{corollary}
Suppose $G=SL_{2}(\qp)$. We adopt the notation established in the beginning of section \ref{sec:example}.  Suppose $\psi, \varphi$ are smooth characters of $T$. Let $k$ be a negative even integer and let $\ell$ be any even integer, where $k\neq \ell$. Then $\Ext^{1}_{G}(\ind_{\overline{P}}^{G}\chi_{k}\psi,I_{\overline{P}}^{G}(\chi_{\ell}\phi))$ is trivial unless $k=-(\ell+2)$ and \begin{itemize}
    \item $\phi\delta_P \neq \phi^{w}$ and $\psi = \phi$ (in which case it is one dimensional),
    \item $\phi\delta_P \neq \phi^{w}$ and $\psi\delta_P=\phi^{w}$ (in which case it is at most one dimensional),
    \item $\phi\delta_P = \phi^{w}$ and $\psi = \phi$ (in which case it is one or two dimensional),
    \item $\psi\delta_P=\phi^{w}$ (in which case it is at most one dimensional).
\end{itemize} 
\end{corollary} 
\begin{proof}
The fact that $k$ is negative, together with example \ref{ex:Jacquet_indchi} show that $U=\chi_{k}\psi$ satisfies the conditions of the previous proposition (also note that $I_{\overline{P}}^{G}(\chi_{k}\psi) = \ind_{\overline{P}}^{G}\chi_{k}\psi$ \cite[Ex. 5.1.9]{Jacquet2}). From example \ref{ex:Jacquet_indchi} and \ref{ex:L(k)}, \[
    H^1J_P(I_{\overline{P}}^{G}(\chi_{-(\ell+2)}\phi)) = \begin{cases} 
    \text{an extension of }\chi_{-(\ell+2)}\phi^{w}\text{ by }\chi_{-(\ell+2)}\phi\delta_P & \text{ if $\ell \geq 0$}\\
    \chi_{-(\ell+2)}\phi^{w} & \text{ if $\ell<0$} 
    \end{cases}.
\] The rest follows from Corollary 8.8 and Theorem 8.9 of \cite{Kohlhaase} along with the observation that $\F_{\overline{P}}^{G}(M(-(k+2)),\psi)$ is a non-trivial class in $\Ext^{1}_{G}(\ind_{\overline{P}}^{G}\chi_{k}\psi,I_{\overline{P}}^{G}(\chi_{-(k+2)}\phi))$. \end{proof}

\begin{proposition}\label{prop:inj_alpha}
Suppose $J_P(I_{\overline{P}}^{G}(U))=U(\delta_{P})$, $U(\g)\otimes_{U(\mathfrak{p})} U$ is irreducible and the identity map on $J_P(V)$ is balanced. Then there is a natural injective map \[
    \alpha: \Ext^{1}_{M}(U(\delta_{P}),J_{P}(V)) \ra \Ext^{1}_{G}(I_{\overline{P}}^{G}(U),V).
\]
\end{proposition}
\begin{proof}
First, we give the construction of the map $\alpha$. Given a class $[C]\in \Ext^{1}_{M}(U(\delta_{P}),J_{P}(V))$, we can twist by $\delta_{P}^{-1}$ and apply the functor $I_{\overline{P}}^{G}$ to obtain a short exact sequence (see Theorem \ref{thm:exact_IpG}) \begin{equation}\label{eq:apply_I}
    0 \ra I_{\overline{P}}^{G}(J_{P}(V)(\delta_{P}^{-1})) \ra I_{\overline{P}}^{G}(C(\delta_{P}^{-1})) \ra I_{\overline{P}}^{G}(U) \ra 0.
\end{equation} Our assumption on $J_P(V)$ ensures that the identity map on $J_P(V)$ corresponds to a $G$-equivariant linear map $I_{\overline{P}}^{G}(J_{P}(V)(\delta_{P}^{-1})) \ra V$ through the adjunction formula (\ref{eq:Em_adjunction}). Pushing forward the exact sequence (\ref{eq:apply_I}) along this map gives us a class in $\Ext^{1}_{G}(I_{\overline{P}}^{G}(U),V)$, which we define to be $\alpha([C])$. 

By applying the Jacquet-Emerton functor $J_P$ on (\ref{eq:apply_I}) and the exact sequence defining $\alpha([C])$, together with Lemma 0.3 of \cite{Jacquet2}, there is a commutative diagram of the form \begin{equation}\label{eq:inj_of_alpha}
\begin{tikzcd}[column sep = tiny]
    0 \arrow[r] & J_P(V)(\delta_P) \arrow[r] \arrow[d] & C(\delta_P) \arrow[r] \arrow[d] & U(\delta_P) \arrow[r] \arrow[d] &  0 \arrow[d] & \\
    0 \arrow[r] & J_P\left(I_{\overline{P}}^{G}(J_P(V)(\delta_P))\right) \arrow[r] \arrow[d] & J_P(I_{\overline{P}}^{G}(C(\delta_P))) \arrow[r] \arrow[d] & U(\delta_P) \arrow[r] \arrow[d] &  H^{1}J_P\left(I_{\overline{P}}^{G}(J_P(V)(\delta_P))\right) \arrow[r] \arrow[d] & ... \\
    0 \arrow[r] & J_P(V)(\delta_P) \arrow[r] & \alpha(C) \arrow[r] & U(\delta_P) \arrow[r] \arrow[l, bend right = 20] &  H^{1}J_{P}(\delta_P)) \arrow[r] & ... \\
\end{tikzcd}.    
\end{equation} Here, the horizontal rows are exact. By construction, the composition of the vertical maps on the very left is the identity map on $J_P(V)(\delta_P)$. The third vertical maps are all identity as well. By five lemma, $\alpha(C)\cong C(\delta)$. Therefore, $\alpha([C])$ splits if and only if $C$ splits. 
\end{proof}

\begin{remark}
Suppose the hypotheses of Proposition \ref{prop:inj_alpha} hold. Define the map $\beta$ in (\ref{eq:edge_morph2}) to be applying the Jacquet-Emerton module functor, as in Proposition \ref{prop:inj_beta}. The proof of Proposition \ref{prop:inj_alpha} also proves that the image of $\alpha$ is contained in the kernel of $\beta$. In order to show the reverse inclusion, we would need to be able to compare $[D]$ and $[I_{\overline{P}}^{G}(J_P(D))]$ for all $[D]$ in the kernel of $\beta$. One possible comparison comes from using the adjunction formula (\ref{eq:Em_adjunction}), but it would require the identity map on $J_P(D)$ to be balanced, which may not be true. 
\end{remark}


\end{document}